\documentclass[10pt]{amsart}

\usepackage{amsmath, amssymb, graphicx}
\usepackage{pst-all}
\usepackage[T1]{fontenc} 

\allowdisplaybreaks

\newtheorem{coro}[equation]{Corollary}
\newtheorem{defi}[equation]{Definition}
\newtheorem{lemm}[equation]{Lemma}
\newtheorem{nota}[equation]{Notation}
\newtheorem{prop}[equation]{Proposition}
\newtheorem{exam}[equation]{Example}
\newtheorem{rema}[equation]{Remark}

\psset{unit=1mm}
\psset{fillcolor=white}
\psset{dotsep=1.5pt}
\psset{dash=1.5pt 1.5pt}
\psset{linewidth=0.8pt}
\psset{arrowsize=3.5pt}
\psset{doublesep=1pt}
\psset{coilwidth=1.2mm}
\psset{coilheight=2}
\psset{coilaspect=20}
\psset{coilarm=2}
\newpsstyle{double}{linewidth=0.5pt, doubleline=true}
\newpsstyle{etc}{linestyle=dotted}
\newpsstyle{exist}{linestyle=dashed}
\newpsstyle{thin}{linewidth=0.5pt}
\newpsstyle{thinexist}{linewidth=0.5pt, linestyle=dashed}
\def\AArrow(#1,#2){\ncline[nodesep=0.3mm, linewidth=0.8pt, border=1.2pt]{->}{#1}{#2}}
\def\BArrow(#1,#2){\ncline[linecolor=blue, linewidth=1.2pt, linestyle=dotted, nodesep=0.3mm, border=1.2pt]{->}{#1}{#2}}
\def\CArrow(#1,#2){\ncline[linecolor=red, nodesep=0.3mm, linewidth=0.8pt, border=1.2pt]{->}{#1}{#2}}

\numberwithin{equation}{section}

\newcounter{ITEM}
\newcommand\ITEM[1]{\setcounter{ITEM}{#1}\leavevmode\hbox{\rm(\roman{ITEM})}}

\newcommand\aaa[1]{a_{#1}}
\newcommand\At{A}
\newcommand\ai[1]{121...[#1]}
\newcommand\aii[1]{212...[#1]}
\newcommand\Att{\widetilde{\VR(2.1,0)\smash{\mathrm{A}}}_2}
\newcommand\BKL[1]{B^{*+}_{#1}}
\newcommand\BP[1]{B^{\scriptscriptstyle+}_{#1}}
\newcommand\BR[1]{B_{#1}}
\newcommand\cc{c}
\newcommand\Col{C_X}
\newcommand\coll{\pi}
\newcommand\Colp{\widehat{C}_X}
\newcommand\comp{\mathbin{\vcenter{\hbox{$\scriptscriptstyle\circ$}}}}
\newcommand\DD{D}
\newcommand\Div{\mathrm{Div}}
\newcommand\dive{\preccurlyeq\nobreak}
\newcommand\divs{\prec\nobreak}
\newcommand\ee{e}
\newcommand\ev{\textsc{ev}}
\newcommand\ff{f}
\newcommand\FF{F}
\newcommand\FFF{L}

\renewcommand\ge{\geqslant}
\renewcommand\gg{g}
\newcommand\GG{G}
\newcommand\GR[2]{\langle#1\mid#2\rangle}
\newcommand\hh{h}
\newcommand\HH{H}
\newcommand\HS[1]{\hspace{#1ex}}
\newcommand\ii{i}
\newcommand\inv{^{-1}}
\newcommand\jj{j}
\newcommand\kk{k}
\renewcommand\le{\leqslant}
\newcommand\Lex{{\scriptscriptstyle\mathrm{Lex}}}
\newcommand\LG[1]{\Vert#1\Vert}
\newcommand\mm{m}
\newcommand\MM{M}
\newcommand\MON[2]{\langle#1\mid#2\rangle^{\!+}}
\newcommand\NF{\textsc{nf}}
\newcommand\NM{N}
\newcommand\NMbar{\begin{picture}(2.5,2.75)(0,0)
\put(0,0){$\NM\HS{-0.2}$}
\psline[linewidth=0.35pt](0.6,2.85)(3.2,2.85)
\end{picture}}
\newcommand\NMD{\NM^{\HS{-0.3}\Delta}}
\newcommand\NMDbar{\begin{picture}(5,3)(0,0)
\put(0,0){$\NMD\HS{-0.2}$}
\psline[linewidth=0.35pt](0.8,3.3)(4.9,3.3)
\end{picture}}
\newcommand\NMDbarr[1]{\begin{picture}(4.8,3)(0,0)
\put(0,0){$\NMD\HS{-0.2}$}
\put(2.7,-1.6){$\scriptstyle#1$}
\psline[linewidth=0.35pt](0.8,3.5)(4.9,3.5)
\end{picture}}
\newcommand\NMS{\NM^{\HS{-0.2}\SSS}}

\newcommand\NMSbarr[1]{\begin{picture}(4.8,3)(0,0)
\put(0,0){$\NMS\HS{-0.2}$}
\put(2.5,-1.2){$\scriptstyle#1$}
\psline[linewidth=0.35pt](0.8,3.2)(4.7,3.2)
\end{picture}}
\newcommand\nn{n}

\newcommand\NNNN{\mathbb{N}}
\newcommand\notdive{\mathrel{\not\dive}}
\newcommand\pdots{\HS{0.2}{\cdot}{\cdot}{\cdot}\HS{0.2}}
\newcommand\Pow[2]{#1^{[#2]}}
\newcommand\pp{p}
\newcommand\PP{P}
\newcommand\perm{\pi}
\newcommand\qq{q}
\newcommand\resp{\mbox{\it resp}.\ }
\newcommand\rr{r}
\newcommand\RR{R}
\newcommand\RRRR{\mathbb{R}}
\newcommand\sep{\HS{0.05}{\vert}\HS{0.05}}
\newcommand\sh{\mathrm{sh}}
\newcommand\sig[1]{\sigma_{\hspace{-0.2ex}#1}^{\null}}
\newcommand\sigg[2]{\sigma_{\hspace{-0.2ex}#1}^{#2}}
\newcommand\siginv[1]{\sigma_{\hspace{-0.2ex}#1}^{-1}}
\newcommand\sss{s}
\newcommand\SSS{S}
\newcommand\Sym{\mathfrak{S}}
\newcommand\ttt{t}
\newcommand\TT{T}
\newcommand\tta{\mathtt{a}}
\newcommand\ttb{\mathtt{b}}
\newcommand\ttc{\mathtt{c}}
\newcommand\ttd{\mathtt{d}}
\newcommand\ttx{\mathtt{x}}
\newcommand\tty{\mathtt{y}}
\newcommand\ttz{\mathtt{z}}
\newcommand\under{\backslash}
\newcommand\uu{u}
\def\VR(#1,#2){\vrule width0pt height#1mm depth#2mm}
\newcommand\vv{v}
\newcommand\wdots{, ...\HS{0.2},}
\newcommand\ww{w}

\newcommand\XX{X}

\newcommand\ZZZZ{\mathbb{Z}}

\title{Garside and quadratic normalisation: a survey}

\author{Patrick Dehornoy}

\address{Laboratoire de Math\'ematiques Nicolas Oresme, CNRS UMR 6139, Universit\'e de Caen, 14032 Caen cedex, France, and Institut Universitaire de France}
\email{patrick.dehornoy@unicaen.fr}
\urladdr{www.math.unicaen.fr/$\sim$dehornoy}



\begin{document}

\maketitle

\begin{abstract}
Starting from the seminal example of the greedy normal norm in braid monoids, we analyze the mechanism of the normal form in a Garside monoid and explain how it extends to the more general framework of Garside families. Extending the viewpoint even more, we then consider general quadratic normalisation procedures and characterise Garside normalisation among them.
\end{abstract}


This text is an essentially self-contained survey of a general approach of normalisation in monoids developed in recent years in collaboration with several co-authors and building on the seminal example of the greedy normal form of braids independently introduced by S.\,Adjan~\cite{Adj} and by M.\,El-Rifai and H.\,Morton~\cite{ElM}. The main references are the book~\cite{Garside}, written with F.\,Digne, E.\,Godelle, D.\,Kram\-mer, and J.\,Michel, the recent preprint~\cite{Dip}, written with Y.\,Guiraud, and, for algorithmic aspects, the article~\cite{Dig}, written with V.\,Gebhardt.

If $\MM$ is a monoid (or a semigroup), and $\SSS$ is a generating subfamily of~$\MM$, then, by definition, every element of~$\MM$ is the evaluation of some $\SSS$-word. A \emph{normal form} for~$\MM$ with respect to~$\SSS$ is a map that assigns to each element of~$\MM$ a distinguished representative $\SSS$-word, hence a (set theoretic) section for the evaluation map from~$\SSS^*$ to~$\MM$. The interest of normal forms is obvious, since they provide a unambiguous way of specifying the elements of~$\MM$ and, from there, for working with them in practice. As can be expected, the complexity of a normal form is a significant element. It can be defined either by considering the complexity of the language of normal words (regular, algebraic, etc.), or that of the associated normalisation map, that is, the procedure that transforms an arbitrary word into an equivalent normal word (linear, polynomial, etc.).

A huge number of normal forms appear in literature, based on quite different initial approaches, and it is certainly difficult to establish nontrivial results unifying all possible types. In this text, we concentrate on some families of normal forms that turn out to be simple in terms of complexity measures, and whose main specificity is to satisfy some \emph{locality} assumptions, meaning that both the property of being normal and the procedure that transforms an arbitrary word into an equivalent normal word only involve factors of a bounded length, here factors of length two (and that is why we call them ``quadratic''). As we shall see, several well-known classes of normalisation processes enter this framework, for instance the seminal example of the greedy normal form in Artin's braid monoids~\cite{Adj,ElM,Eps} but also the normal forms in Artin--Tits monoids based on rewriting systems as in~\cite{GaussentGuiraudMalbos} or, in a very different context, the normal form in plactic monoids based on Young tableaux and the Robinson--Schenstedt algorithm~\cite{BokutChenChenLi,CainGrayMalheiro}. 

The current introductory text is organised in four sections, going from the particular to the general. In Sec.~\ref{S:NormalE}, we analyze two motivating examples of greedy normal forms, involving free abelian monoids, a toy case that already contains the main ideas, and braid monoids, a more complicated case. Extending these examples, we describe in Sec.\,\ref{S:NormalD} the mechanism of the $\Delta$-normal form in the now classical framework of Garside monoids. Next, in Sec.\,\ref{S:NormalS}, we explain how most of the results can be generalized and, at the same time, simplified, using the notion of an $\SSS$-normal form derived from a Garside family. Finally, in Sec.\,\ref{S:NormalQ}, we introduce quadratic normalisations, which provide a natural unifying framework for the normal forms we are interested in. Having defined a complexity measure called the class, we characterise Garside normalisations among quadratic normalisations, and mention (positive and negative) termination results for the rewriting systems naturally associated with quadratic normalisations.

Most proofs are omitted or only sketched. However, it turns out that some arguments, mainly in Sec.\,\ref{S:NormalD} and~\ref{S:NormalS}, are very elementary, and then we included them, hopefully making this text both more informative and thought-provoking.

Excepte in concluding remarks at the end of sections, we exclusively consider monoids. A number of results, in particular most of those involving Garside normalisation, can be extended to groups of fractions. Also, the whole approach extends to categories, viewed as monoids with a partial multiplication.

Our notation is standard. We use $\ZZZZ$ for the set of integers, $\NNNN$ for the set of nonnegative integers. If $\SSS$ is a set, we denote by~$\SSS^*$ the free monoid over~$\SSS$ and call its elements \emph{$\SSS$-words}, or simply \emph{words}. In this context, $\SSS$ is also called alphabet, and its elements letters. We write $\LG\ww$ for the length of~$\ww$, and $\Pow\SSS\pp$ for the set of all $\SSS$-words of length~$\pp$. We use $\ww \sep \ww'$ for the concatenation of two $\SSS$-words~$\ww$ and~$\ww'$. We say that $\ww'$ is a \emph{factor} of~$\ww$ if there exist~$\uu, \vv$ satisfying $\ww = \uu \sep \ww' \sep \vv$. If $\MM$ is a monoid generated by a set~$\SSS$, we say that an $\SSS$-word $\sss_1 \sep \pdots \sep \sss_\pp$ is a \emph{$\SSS$-decomposition} for an element~$\gg$ of~$\MM$ if $\gg = \sss_1 \pdots \sss_\pp$ holds in~$\MM$. 

\section{Two examples}\label{S:NormalE}

We shall describe a specific type of normal form often called the ``greedy normal form''. Before describing it in full generality, we begin here with two examples: a very simple one involving free abelian monoids, and then the seminal example of Artin's braid monoids as investigated after Garside.

\subsection{Free abelian monoids}\label{SS:Abelian}

Our first example, free abelian monoids, is particularly simple, but it is fundamental as it can serve as a model for the sequel: our goal will be to obtain for more complicated monoids counterparts of the results that are trivial here. 

Consider the free abelian monoid~$(\NNNN, +)^\nn$ with $\nn \ge 1$, simply denoted by~$\NNNN^\nn$. We shall see the elements of~$\NNNN^\nn$ as sequences of nonnegative integers indexed by $\{1\wdots \nn\}$, thus writing $\gg(\kk)$ for the $\kk$th entry of an element~$\gg$, and use a multiplicative notation: $\ff\gg = \hh$ means $\forall\kk\,(\ff(\kk) + \gg(\kk) = \hh(\kk))$. Let $\At_\nn$ be the family $\{\tta_1 \wdots \tta_\nn\}$, where $\tta_\ii$ is defined by $\tta_\ii(\kk) = 1$ for $\kk = \ii$, and $0$ otherwise. Then $\At_\nn$ is a basis of~$\NNNN^\nn$ as an abelian monoid.

It is straightforward to obtain a normal form~$\NF^\Lex$ for~$\NNNN^\nn$ with respect to~$\At_\nn$ by fixing a linear ordering on~$\At_\nn$, for instance $\tta_1 < \pdots < \tta_\nn$, and, for~$\gg$ in~$\NNNN^\nn$, defining $\NF^\Lex(\gg)$ to be the lexicographically smallest word representing~$\gg$.

We shall be interested here in another normal form, connected with another generating family. In this basic example, the construction may seem uselessly intricate, but we shall see that it nicely extends to less trivial cases, which is not the case of the above lexicographical normal form. Let us put
\begin{equation}\label{E:AbelianCube}
\SSS_\nn:= \{ \sss \in \NNNN^\nn \mid \sss(\kk) \in \{0,1\} \text{\ for $\kk = 1 \wdots \nn$}\}.
\end{equation}
For $\ff, \gg$ in~$\NNNN^\nn$, define $\ff \dive \gg$ to mean $\forall\kk\,(\ff(\kk) \le \gg(\kk))$, and write $\ff \divs \gg$ for $\ff \dive \gg$ with $\ff \not= \gg$. The relation~$\dive$ is a partial order, connected with the operation of~$\NNNN^\nn$, since $\ff \dive \gg$ is equivalent to $\exists\gg'{\in}\NNNN^\nn\, (\ff\gg' = \gg)$, that is, $\ff$ \emph{div\-ides}~$\gg$ in~$\NNNN^\nn$. Then $\SSS_\nn$ consists of the $2^\nn$~divisors of the element~$\Delta_\nn$ defined by 
\begin{equation}\label{E:AbelianDelta}
\Delta_\nn(\kk) := 1 \text{\quad for $\kk = 1 \wdots \nn$}.
\end{equation}

We recall that, if $\MM$ is a (left-cancellative) monoid generated by a family~$\SSS$, the \emph{Cayley graph of~$\MM$ with respect to~$\SSS$} is the $\SSS$-labeled oriented graph with vertex set~$\MM$ and, for~$\gg, \hh$ in~$\MM$ and $\sss$ in~$\SSS$, there is an $\sss$-labeled edge from~$\gg$ to~$\hh$ if, and only if, $\gg \sss = \hh$ holds in~$\MM$. Then, the Cayley graph of~$\NNNN^\nn$ with respect to~$\At_\nn$ is an $\nn$-dimensional grid, and $\SSS_\nn$ corresponds to the $\nn$-dimensional cube that is the elementary tile of the grid, see Fig.\,\ref{F:Abelian}. 

\begin{prop}\label{P:AbelianNormal}
Every element of~$\NNNN^\nn$ admits a unique decomposition of the form $\sss_1 \sep \pdots \sep \sss_\pp$ with $\sss_1 \wdots \sss_\pp$ in~$\SSS_\nn$ satisfying $\sss_\pp \not= 1$, and, for every~$\ii <\pp$,
\begin{equation}\label{E:AbelianNormal}
\forall \sss {\in}Ê\SSS_\nn \, (\, \sss_\ii \divs \sss \ \Rightarrow\ \sss \notdive \sss_\ii\sss_{\ii+1}\pdots \sss_\pp\, ).
\end{equation}
\end{prop}

Condition~\eqref{E:AbelianNormal} is a maximality statement. It says that $\sss_1$ contains as much of~$\gg$ as possible in order to remain in~$\SSS_\nn$ and that, for every~$\ii$, the entry~$\sss_\ii$ similarly contains as much of the right chunk $\sss_\ii \pdots \sss_\pp$ as possible to remain in~$\SSS_\nn$: when we try to replace~$\sss_\ii$ with a larger element~$\sss$ of~$\SSS_\nn$, then we quit the divisors of~$\sss_\ii \pdots \sss_\pp$. This should make it clear why the decomposition $\sss_1 \sep \pdots \sep \sss_\pp$ of~$\gg$ is usually called \emph{greedy}. 

\begin{exam}\label{X:Abelian}
\rm Assume $\nn = 3$ and consider $\gg = \tta^3 \ttb \ttc^2$, that is, $\gg = (3,1,2)$ (we write $\tta, \ttb, \ttc$ for $\tta_1, \tta_2, \tta_3$). The maximal element of~$\SSS_3$ that divides~$\gg$ is~$\Delta_3$, with $\gg = \Delta_3 \cdot \tta^2 \ttc$. Then, the maximal element of~$\SSS_3$ that divides~$\tta^2 \ttc$ is~$\tta \ttc$, with $\tta^2 \ttc = \tta \ttc \cdot \tta$. The latter element left-divides~$\Delta_3$. So the greedy decomposition of~$\gg$ as provided by Prop.\,\ref{P:AbelianNormal} is the length-three $\SSS_3$-word $\Delta_3 \sep \tta \ttc \sep \tta$, see Fig.\,\ref{F:Abelian}.
\end{exam}

\begin{figure}[htb]\centering
\begin{picture}(81,37)(0,-1)
\put(0,0){\includegraphics[scale=0.4]{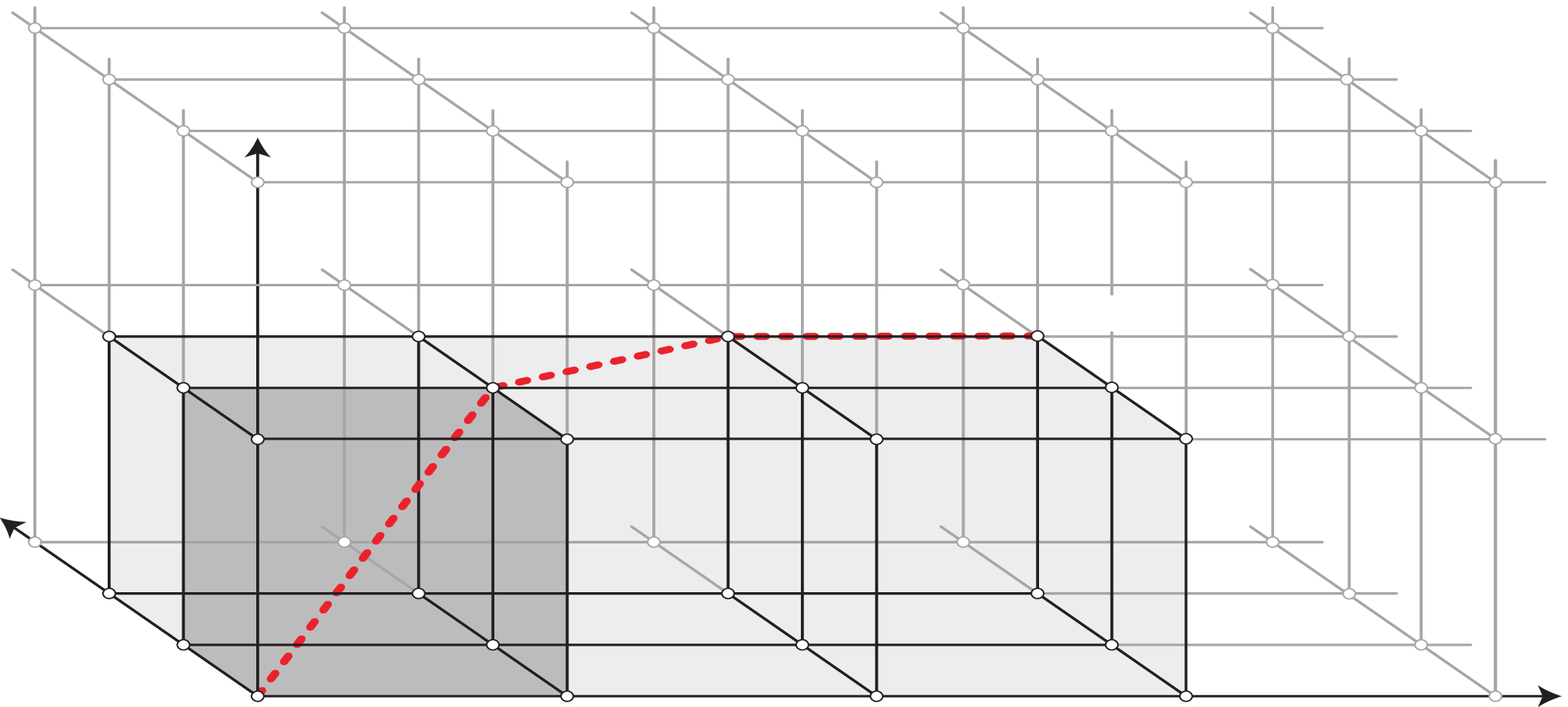}}
\put(11.5,-1.5){$1$}
\put(30,-1.5){$\tta$}
\put(7.5,18.5){$\ttb$}
\put(7,2){$\ttc$}
\put(25,18.5){$\Delta$}
\put(54,20){$\tta^3 \ttb \ttc^2$}
\end{picture}
\caption{The Cayley graph of $\NNNN^\nn$ with respect to~$\At_\nn$, here for $\nn = 3$; we write $\tta, \ttb, \ttc$ for $\tta_1, \tta_2, \tta_3$, and $\Delta$ for~$\Delta_3$. The dark grey cube corresponds to the $8$ elements of~$\SSS_3$. Then, the greedy decomposition of $\tta^3 \ttb \ttc^2$ corresponds to the dashed path: among the many possible ways of going from~$1$ to~$\tta^3 \ttb \ttc^2$, we choose at each step the largest possible element of~$\SSS_3$ that divides the considered element, thus remaining in the light grey domain, which corresponds to the divisors of $\tta^3 \ttb \ttc^2$.}
\label{F:Abelian}
\end{figure}

Prop.\,\ref{P:AbelianNormal} is easy. It can be derived from the following (obvious) observation: 

\begin{lemm}\label{L:AbelianLattice}
For every~$\nn$, the divisibility relation of~$\NNNN^\nn$ is a lattice order, and $\SSS_\nn$ is a finite sublattice formed by the divisors of~$\Delta_\nn$, which are $2^\nn$ in number.
\end{lemm}

We recall that a \emph{lattice order} is a partial order in which every pair of elements admits a greatest lower bound and a lowest upper bound. When considering a divisibility relation, it is natural to use ``least common multiple'' (\emph{lcm}) and ``greatest common divisor'' (\emph{gcd}) for the least upper and greatest lower bounds. 

Once Lemma~\ref{L:AbelianLattice} is available, Prop.\,\ref{P:AbelianNormal} easily follows: indeed, starting from~$\gg$, if $\gg$ is not~$1$, there exists a maximal element~$\sss_1$ of~$\SSS_\nn$ dividing~$\gg$, namely the left-gcd of~$\gg$ and~$\Delta_\nn$. So there exists~$\gg'$ satisfying $\gg = \sss_1 \gg'$. If $\gg'$ is not~$1$, we repeat with~$\gg'$, finding a maximal element~$\sss_2$ of~$\SSS_\nn$ dividing~$\gg'$, etc. The sequence $\sss_1 \sep \sss_2 \sep \pdots$ so obtained then satisfies the maximality condition of~\eqref{E:AbelianNormal}.

\begin{rema}
\rm Let $\ZZZZ^\nn$ be the rank~$\nn$ free abelian group. Then $\ZZZZ^\nn$ is a group of (left) fractions for the monoid~$\NNNN^\nn$, meaning that every element of~$\ZZZZ^\nn$ admits an expression $\ff\inv\gg$ with $\ff, \gg$ in~$\NNNN^\nn$. It is easy to extend the greedy normal form of Prop.\,\ref{P:AbelianNormal} into a unique normal form on the group~$\ZZZZ^\nn$: indeed, every element of~$\ZZZZ^\nn$ can be expressed as $\Delta_\nn^\mm \gg$ with $\mm$ in~$\ZZZZ$ and~$\gg$ in~$\NNNN^\nn$, hence it admits a decomposition $\Delta_\nn^\mm \sep \sss_1 \sep \pdots \sep \sss_\pp$ with $\sss_1 \wdots \sss_\pp$ in~$\SSS_\nn$ satisfying~\eqref{E:AbelianNormal}. The latter is not unique in general, but it is if, in addition, one requires $\sss_1 \not= \Delta_\nn$.
\end{rema}

\subsection{Braid monoids}\label{SS:Braids}

Much less trivial, our second example involves braid monoids as investigated by F.A.\,Garside in~\cite{Gar}. For our current purpose, it is convenient to start with a presentation of the braid monoid~$\BP\nn$, namely
\begin{equation}\label{E:BraidPresent}
\BP\nn = \bigg\langle \sig1 \wdots \sig{n-1} \ \bigg\vert\ 
\begin{matrix}
\sig\ii \sig j = \sig j \sig\ii 
&\text{for} &\vert i-j \vert\ge 2\\
\sig\ii \sig j \sig\ii = \sig j \sig\ii \sig j \ 
&\text{for} &\vert i-j \vert = 1
\end{matrix}
\ \smash{\bigg\rangle^{\!\!+}}.
\end{equation}
The braid group~$\BR\nn$ is the group which, as a group, admits the presentation~\eqref{E:BraidPresent}. As shown by E.\,Artin in~\cite{Art} (see, for instance, \cite{Bir} or \cite{Dhr}), $\BR\nn$ interprets as the group of isotopy classes of $\nn$-strand braid diagrams, which are planar diagrams obtained by concatenating diagrams of the type
$$\begin{picture}(42,19)(-5,1)
\psline[linewidth=0.8mm]{c-c}(0,0)(0,6)
\psline[linewidth=0.8mm]{c-c}(12,0)(12,6)
\psline[linewidth=0.8mm]{c-c}(18,6)(24,0)
\psline[linewidth=0.8mm,border=2pt]{c-c}(24,6)(18,0)
\psline[linewidth=0.8mm]{c-c}(30,0)(30,6)
\psline[linewidth=0.8mm]{c-c}(42,0)(42,6)
\psline[linewidth=0.8mm]{c-c}(0,10)(0,16)
\psline[linewidth=0.8mm]{c-c}(12,10)(12,16)
\psline[linewidth=0.8mm]{c-c}(24,16)(18,10)
\psline[linewidth=0.8mm,border=2pt]{c-c}(18,16)(24,10)
\psline[linewidth=0.8mm]{c-c}(30,10)(30,16)
\psline[linewidth=0.8mm]{c-c}(42,10)(42,16)
\put(-12,12){$\sig\ii\,:$}
\put(-24,3){and\quad $\siginv\ii\,:$}
\put(-1,18){$1$}
\put(17,18){$\ii$}
\put(22,18){$\ii{+}1$}
\put(41,18){$\nn$}
\put(4.5,2.5){$\pdots$}
\put(34.5,2.5){$\pdots$}
\put(4.5,12.5){$\pdots$}
\put(34.5,12.5){$\pdots$}
\end{picture}$$
with $1 \le \ii \le \nn-1$. When an $\nn$-strand braid diagram is viewed as the projection of $\nn$~nonintersecting curves in a cylinder~$\DD^2 \times \RRRR$ as in Fig.\,\ref{F:Braid}, the relations of~\eqref{E:BraidPresent} correspond to the natural notion of a deformation, or \emph{ambient isotopy}. Then, the monoid~$\BP\nn$ corresponds to isotopy classes of positive braid diagrams, meaning those diagrams in which all crossings have the orientation of~$\sig\ii$.

\begin{figure}[htb]\centering
\begin{picture}(112,27)(0,1)
\put(0,0){\includegraphics[scale=0.6]{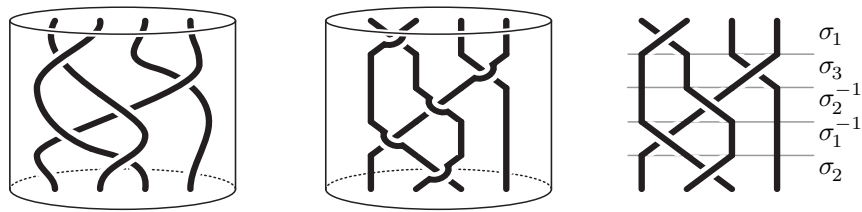}}
\put(108,23){$\sig1$}
\put(108,18.5){$\sig3$}
\put(108,14){$\siginv2$}
\put(108,9.5){$\siginv1$}
\put(108,5){$\sig2$}
\end{picture}
\caption{Viewing an $\nn$-strand braid diagram (here $\nn = 4$) as the plane projection of a 3D-figure in a cylinder; on the right, by decomposing the diagram into elementary diagrams involving only one crossing, with two possible orientations, one obtains an encoding of an $\nn$-braid diagram by a word in the alphabet $\{\sigg1{\pm1} \wdots \sigg{\nn{-}1}{\pm1}\}$.}
\label{F:Braid}
\end{figure}

Defining unique normal forms for the elements of the monoid~$\BP\nn$ (called \emph{positive $\nn$-strand braids}) is both easy and difficult. Indeed, by definition, every positive $\nn$-strand admits decompositions in terms of the letters~$\sig1 \wdots \sig{\nn{-}1}$, and, as in Subsec.\,\ref{SS:Abelian}, we obtain a distinguished expression by considering the lexicographically smallest expression. This, however, is \emph{not} a good idea: the normal form so obtained is almost useless (nevertheless, see \cite{AlNa}, building on unpublished work by Bronfman, for an application in combinatorics), mainly because there is no simple rule for obtaining the normal form of~$\sig\ii \gg$ or $\gg\sig\ii$ from that of~$\sig\ii$. In other words, one cannot \emph{compute} the normal form easily.

A much better normal form can be obtained as follows. For each~$\nn$, let $\Delta_\nn$ be the positive $\nn$-strand braid inductively defined by
\begin{equation}
\label{E:Delta}
\Delta_1 := 1, \quad \Delta_\nn := \Delta_{\nn-1} \, \sig{\nn-1} \pdots \sig2 \sig1 
\mbox{\quad for $\nn \ge 2$}, 
\end{equation}
corresponding to a (positive) half-turn of the whole family of $\nn$~strands:
$$\begin{picture}(98,18)(0,1)
\psset{xunit=1mm}\psset{yunit=0.53mm}
\pscircle[fillstyle=solid,fillcolor=black](0,18){0.5}
\psline[linewidth=0.8mm]{c-c}(20,21)(26,15)
\psline[linewidth=0.8mm,border=2pt]{c-c}(26,21)(20,15)
\psline[linewidth=0.8mm]{c-c}(46,27)(58,15)(58,9)
\psline[linewidth=0.8mm,border=2pt]{c-c}(52,27)(46,21)(46,15)(52,9)
\psline[linewidth=0.8mm,border=2pt]{c-c}(58,27)(58,21)(46,9)
\psline[linewidth=0.8mm]{c-c}(78,36)(90,24)(90,18)(96,12)(96,0)
\psline[linewidth=0.8mm,border=2pt]{c-c}(84,36)(78,30)(78,24)(84,18)(84,12)(90,6)(90,0)
\psline[linewidth=0.8mm,border=2pt]{c-c}(90,36)(90,30)(78,18)(78,6)(84,0)
\psline[linewidth=0.8mm,border=2pt]{c-c}(96,36)(96,18)(78,0)
\put(2,10){$\Delta_1$}
\put(28,10){$\Delta_2$}
\put(60,13){$\Delta_3$}
\put(98,16){$\Delta_4$}
\end{picture}$$ 
Next, let us call a positive braid \emph{simple} if it can be represented by a positive diagram in which any two strands cross at most once. One shows that the latter property does not depend on the choice of the diagram. By definition, the trivial braid~$1$, and every braid~$\sig\ii$ is simple. We see above that $\Delta_\nn$ is also simple. Let~$\SSS_\nn$ be the family of all simple $\nn$-strand braids.

As in Subsec.\,\ref{SS:Abelian}, let $\dive$ be the \emph{left}-divisibility relation of the monoid~$\BP\nn$: so $\ff \dive \gg$ holds if, and only if, there exists~$\gg'$ (in~$\BP\nn$) satisfying $\ff\gg' = \gg$. For $\nn \ge 3$, the monoid~$\BP\nn$ is not abelian, so left-divisibility does not coincide in general with right-divisibility, defined symmetrically by $\exists\gg'\,(\gg'\ff = \gg)$.

Then, we have the following counterpart of Lemma~\ref{L:AbelianLattice}:

\begin{lemm}{\rm\cite{Gar}}\label{L:BraidLattice}
For every~$\nn$, the left-divisibility relation of the mon\-oid~$\BP\nn$ is a lattice order, and $\SSS_\nn$ is a finite sublattice formed by the left-divisors of~$\Delta_\nn$, which are $\nn!$ in number.
\end{lemm}

Contrasting with Lemma~\ref{L:AbelianLattice}, the proof of Lemma~\ref{L:BraidLattice} is far from trivial. Why $\nn!$ appears here is easy to explain. Every $\nn$-strand braid~$\gg$ induces a well-defined permutation~$\perm(\gg)$ of~$\{1 \wdots \nn\}$, where $\perm(\gg)(\ii)$ is the initial position of the strand that finishes at position~$\ii$ in any diagram representing~$\gg$. In this way, one obtains a surjective homomorphism from~$\BR\nn$ to the symmetric group~$\Sym_\nn$. It turns out that, for every permutation~$\ff$ of~$\{1 \wdots \nn\}$, there exists exactly one simple $\nn$-strand braid whose permutation is~$\ff$: so simple braids (also called ``permutation braids'') provide a (set-theoretic) section for the projection of~$\BR\nn$ to~$\Sym_\nn$, and they are $\nn!$ in number, see Fig.\,\ref{F:BraidLattice}.

\begin{figure}[htb]
$$\begin{picture}(60,43)(0,2)
\put(0,0){\includegraphics[scale=0.35]{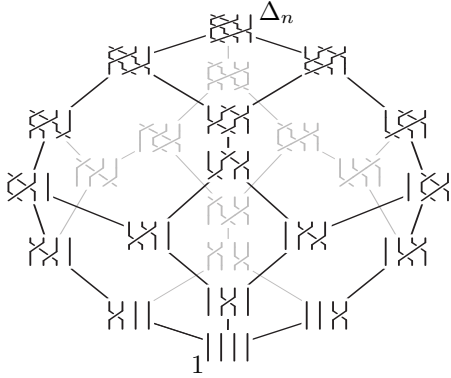}}
\put(25,-1){$1$}
\put(34,46){$\Delta_\nn$}
\end{picture}$$
\caption{The lattice $(\Div(\Delta_\nn), \dive)$ formed by the $\nn!$ left-divisors of the braid~$\Delta_\nn$ in the monoid~$\BP\nn$, here in the case $\nn = 4$: a copy of the $\nn$-permutahedron associated with the symmetric group~$\Sym_\nn$ equipped with what is called the weak order, see for instance~\cite{BjBr}; topologically, this is an $\nn{-}2$-sphere tessellated by hexagons and squares which correspond to the relations of~\eqref{E:BraidPresent}.}
\label{F:BraidLattice}
\end{figure}

Once Lemma~\ref{L:BraidLattice} is available, repeating the argument of Subsec.\,\ref{SS:Abelian} (with some care) easily leads to 

\begin{prop} {\rm\cite{Adj,ElM}}\label{P:BraidNormal}
Every element of~$\BP\nn$ admits a unique decomposition of the form $\sss_1 \sep \pdots \sep \sss_\pp$ with $\sss_1 \wdots \sss_\pp$ in~$\SSS_\nn$ satisfying $\sss_\pp \not= 1$, and, for every~$\ii <\pp$,
\begin{equation}\label{E:BraidNormal}
\forall \sss {\in}Ê\SSS_\nn \, (\, \sss_\ii \divs \sss \ \Rightarrow\ \sss \notdive \sss_\ii\sss_{\ii+1}\pdots \sss_\pp\, ).
\end{equation}
\end{prop}

\enlargethispage{6mm}

In other words, we obtain for every positive braid a unique greedy decomposition exactly similar to the one of Prop.\,\ref{P:AbelianNormal}. 

\begin{exam}\label{X:BraidNormalization}
\rm Consider $\gg = \sig2 \sig3 \sig2\sig2 \sig1 \sig2 \sig3 \sig3$ in~$\BP4$. First, by cutting when two strands that already crossed are about to cross for the second time, we obtain a decomposition into then three simple chunks  $\sig2 \sig3 \sig2 \sep \sig2 \sig1 \sig2 \sig3 \sep \sig3$. Next, we push the crossings upwards as much as possible, we obtain $\sig2 \sig3 \sig2 \sig1 \sep \sig2 \sig1 \sig3 \sep \sig3$, and finally $\sig1 \sig2 \sig3 \sig2 \sig1 \sep \sig2 \sig1 \sig3$, as in the diagram below. with only two entries. 
$$\begin{picture}(96,21)(0,1)
\put(0,0){\includegraphics[scale=0.6]{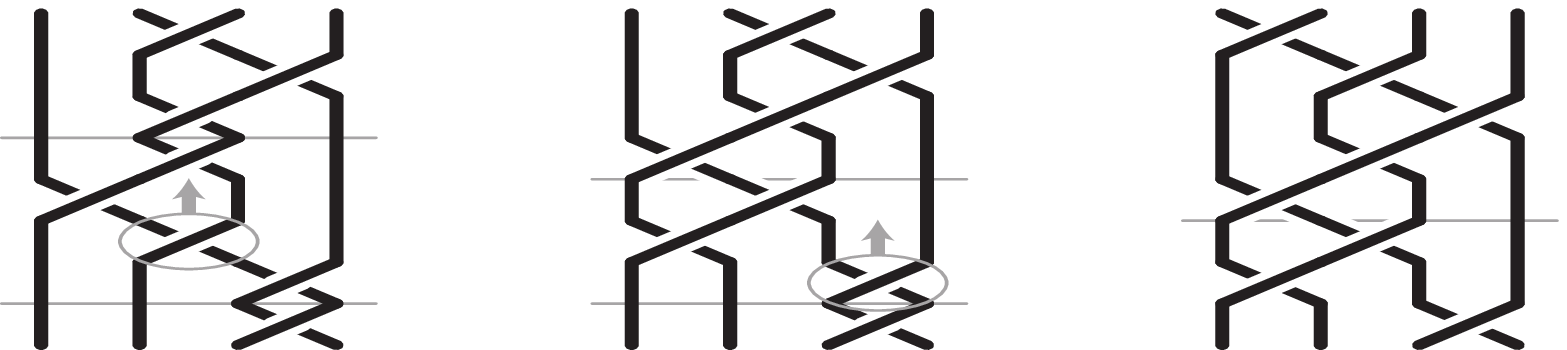}}
\end{picture}$$
We cannot go farther, so the decomposition is greedy.
\end{exam}

\begin{rema}
\rm Here again, the greedy normal form extends from the monoid to the group. It turns out that $\BR\nn$ is a group of fractions for~$\BP\nn$, and that $\Delta_\nn$ is a sort of universal denominator for the elements of the group, meaning that every element of~$\BR\nn$ can be expressed as $\Delta_\nn^\mm \gg$ with $\mm$ in~$\ZZZZ$ and~$\gg$ in~$\BP\nn$. As above, it follows that every element of~$\BR\nn$ admits a unique decomposition $\Delta_\nn^\mm \sep \sss_1 \sep \pdots \sep \sss_\pp$ with $\mm$ in~$\ZZZZ$, $\sss_1 \wdots \sss_\pp$ in~$\SSS_\nn$ satisfying~\eqref{E:AbelianNormal} and, in addition, $\sss_1 \not= \Delta_\nn$. 
\end{rema}

\section{The $\Delta$-normal form in a Garside monoid}\label{S:NormalD}

In the direction of more generality, we now explain how to unify the examples of Sec.\,\ref{S:NormalE} into the notion of a $\Delta$-normal form associated with a Garside element in what is now classically called a Garside monoid.

\subsection{Garside monoids}\label{SS:GarMon}

The greedy normal form of the braid monoid~$\BP\nn$ has been extended to other similar monoids a long time ago. Typically, an \emph{Artin--Tits monoid}, which, by definition, is a monoid defined by relations of the form $\sss\ttt\sss\ttt... = \ttt\sss\ttt\sss...$ where both sides have the same length, is called \emph{spherical} if the Coxeter group obtained by making all generators involutive, that is, adding $\sss^2 = 1$ for each generator~$\sss$, is finite. For instance, \eqref{E:BraidPresent} shows that $\BP\nn$ is an Artin--Tits monoid, whose associated Coxeter group is the finite group~$\Sym_\nn$, so $\BP\nn$ is spherical. Then, building on~\cite{BrS}, it was shown in~\cite{Cha} that all properties of the greedy normal form of braid monoids extend to spherical Artin--Tits monoids.

A further extension came with the notion of a \emph{Garside monoid} (and of a \emph{Garside group}) introduced in~\cite{Dfx} and slightly generalised in~\cite{Dgk}. Recall that a monoid~$\MM$ is  \emph{left-cancellative} (\resp \emph{right-cancellative}) if $\ff\gg = \ff\gg'$ (\resp $\gg\ff = \gg'\ff$) implies $\gg = \gg'$, and \emph{cancellative} if it is both left- and right-cancellative. As above, for $\ff, \gg$ in~$\MM$, we say that $\ff$ is a \emph{left-divisor} of~$\gg$, or, equivalently, that $\gg$ is a \emph{right-multiple} of~$\ff$, written $\ff \dive \gg$, if there exists~$\gg'$ in~$\MM$ satisfying~$\ff\gg' = \gg$. We write $\ff \divs \gg$ for $\ff \dive \gg$ with $\gg \not\dive \ff$ (which amounts to $\ff \not= \gg$ if $\MM$ has no nontrivial invertible element). For~$\gg'\ff = \gg$, we symmetrically say that $\ff$ is a \emph{right-divisor} of~$\gg$, or, equivalently, that $\gg$ is a \emph{left-multiple} of~$\ff$. Note that the set~$\gg \MM$ of all right-multiples of~$\gg$ is the right-ideal generated by~$\gg$, and, similarly, the set~$\MM\gg $ of all right-multiples of~$\gg$ is the left-ideal generated by~$\gg$, as involved in the definition of Green's relations of~$\MM$, see for instance~\cite{How}.

\pagebreak

\begin{defi}\label{D:GarMon}
\rm A \emph{Garside monoid} is a pair~$(\MM, \Delta)$, where $\MM$ is a cancellative monoid satisfying the following conditions:

\ITEM1 There exists $\lambda : \MM \to \NNNN$ satisfying, for all~$\ff, \gg$, 
$$\lambda(\ff \gg) \ge \lambda(\ff) + \lambda(\gg) \text{\qquad and\qquad } \gg \not= 1 \Rightarrow \lambda(\gg) \not= 0.$$

\ITEM2 Any two elements of~$\MM$ admit left- and right-lcms and gcds. 

\ITEM3 $\Delta$ is a \emph{Garside element} of~$\MM$, meaning that the left- and right-divisors of~$\Delta$ coincide and generate~$\MM$, 

\ITEM4 The family~$\Div(\Delta)$ of all divisors of~$\Delta$ in~$\MM$ is finite.
\end{defi}

Note that condition~\ITEM1 in Def.\,\ref{D:GarMon} implies in particular that the monoid~$\MM$ has no nontrivial invertible element (meaning: not equal to~$1$): indeed, $\lambda(1) = \lambda(1 \cdot 1) \ge \lambda(1) + \lambda(1)$ implies $\lambda(1) = 0$, so $\ff\gg = 1$ implies $0 \ge \lambda(\ff) + \lambda(\gg)$, whence $\ff = \gg = 1$. 

\begin{exam}
\rm For every~$\nn$, the pair~$(\NNNN^\nn, \Delta_\nn)$, with $\Delta_\nn$ as defined in~\eqref{E:AbelianDelta}, is a Garside monoid. Indeed, $\NNNN^\nn$ is cancellative, we can define $\lambda(\gg)$ to be the common length of all $\At_\nn$-words representing an element~$\gg$ of~$\NNNN^\nn$ and, according to Lemma~\ref{L:AbelianLattice}, the left- and right-divisibility relations (which coincide since $\NNNN^\nn$ is abelian) are lattice orders. Finally, $\Delta_\nn$ is a Garside element, since its divisors include~$\At_\nn$, and the family~$\Div(\Delta_\nn)$ is finite, since it has $2^\nn$~elements.

Similarly, the pairs $(\BP\nn, \Delta_\nn)$, with $\Delta_\nn$ now defined by~\eqref{E:Delta}, is a Garside monoid as well. That $\BP\nn$ is cancellative is proved in~\cite{Gar}, for~$\lambda(\gg)$ we can take again the common length of all braid words representing~$\gg$, and Lemma~\ref{L:BraidLattice} provides the remaining conditions.
\end{exam}

In the same vein, it can be shown that, if $\MM$ is a spherical Artin--Tits monoid and $\Delta$ is the lifting of the longest element~$\ww_0$ of the associated finite Coxeter group, then $(\MM, \Delta)$ is a Garside monoid. Actually, many more examples are known. Let us mention two. 

\begin{exam}
\rm First, let~$\BKL\nn$, the \emph{dual braid monoid}, be the submonoid of the braid group~$\BR\nn$ generated by the $\nn(\nn-1)/2$ braids 
$$\aaa{\ii , \jj} := \sig{\jj-1} \pdots \sig{\ii+1}\sig\ii\siginv{\ii+1} \pdots \siginv{\jj-1} \qquad \text{with $1 \le \ii < \jj \le \nn$},$$
and let $\Delta_\nn^*:= \aaa{1,2} \aaa{2,3} \pdots \aaa{\nn{-}1, \nn}$ \cite{BKL}. Then $(\BKL\nn, \Delta_\nn^*)$ is a Garside monoid, and $\BR\nn$ is its group of fractions (which shows that a group may be the group of fractions of several Garside monoids). Note that $\BKL\nn$ includes~$\BP\nn$, since $\sig\ii = \aaa{\ii, \ii+1}$ holds. The inclusion is strict for $\nn \ge 3$, since $\aaa{1,3}$ is not a positive braid. The lattice of the divisors of~$\Delta_\nn^*$ has $\frac{1}{\nn+1}{2\nn \choose \nn}$ elements, which are in one--one correspondence with the noncrossing partitions of~$\{1 \wdots \nn\}$~\cite{BDM}.

Second,  for $\nn \ge 1$ and $\ee_1 \wdots \ee_\nn \ge 2$, let
$$\TT^+_{\ee_1 \wdots \ee_\nn} := \MON{\tta_1 \wdots \tta_\nn}{\tta_1^{\ee_1} = \tta_2^{\ee_2} = \pdots = \tta_\nn^{\ee_\nn}}.$$
Define~$\Delta$ to be the common value of~$\tta_\ii^{\ee_\ii}$ for all~$\ii$. Then $(\TT^+_{\ee_1 \wdots \ee_\nn}, \Delta)$ is a Garside monoid. The lattice $\Div(\Delta)$ has $\ee_1 + \pdots + \ee_\nn - \nn + 2$ elements and it consists of~$\nn$ disjoint chains of lengths~$\ee_1 \wdots \ee_\nn$ connecting~$1$ to~$\Delta$. 
\end{exam}

\pagebreak

As one can expect, a greedy normal form exists in every Garside monoid:

\begin{prop}\label{P:NormalD}
Assume that $(\MM, \Delta)$ is a Garside monoid. Say that a $\Div(\Delta)$-word $\sss_1 \sep \pdots \sep \sss_\pp$ is \emph{$\Delta$-normal} if, for every~$\ii <\nobreak \pp$, we have
\begin{equation}\label{E:NormalD}
\forall \sss {\in}Ê\Div(\Delta) \  (\, \sss_\ii \divs \sss \ \Rightarrow\ \sss \notdive \sss_\ii\sss_{\ii+1}\pdots \sss_\pp\, ),
\end{equation}
and that it is \emph{strict} if, in addition, $\sss_\pp \not= 1$ holds. Then every element of~$\MM$ admits a unique strict $\Delta$-normal decomposition.
\end{prop}

\begin{proof}
We show, using induction on~$\ell$, that every element~$\gg$ of~$\MM$ satisfying $\lambda(\gg) \le \ell$ admits a strict $\Delta$-normal decomposition. For $\ell = 0$, the only possibility is $\gg = 1$, and then the empty sequence is a $\Delta$-normal decomposition of~$\gg$. Assume $\ell \ge 1$, and let~$\gg$ satisfy $\lambda(\gg) \le \ell$. The case $\gg = 1$ has already been considered, so we can assume $\gg \not= 1$. Let $\sss_1$ be the left-gcd of~$\gg$ and~$\Delta$. As $\Div(\Delta)$ generates~$\MM$, some nontrivial divisor of~$\Delta$ must left-divide~$\gg$, so $\sss_1 \not= 1$ holds. As $\MM$ is left-cancellative, there is a unique element~$\gg'$ satisfying $\gg = \sss_1 \gg'$. By assumption, one has $\lambda(\gg') \le \lambda(\gg) - \lambda(\sss_1) < \ell$. Then, by the induction hypothesis, $\gg'$ admits a strict $\Delta$-normal decomposition $\sss_2 \sep \pdots \sep \sss_\pp$. Then one easily checks that $\sss_1 \sep \sss_2 \sep \pdots \sep \sss_\pp$ is a strict $\Delta$-normal decomposition for~$\gg$. 

As for uniqueness, it is easy to see that the first entry of any greedy decomposition of~$\gg$ must be~$\sss_1$, and then we apply the induction hypothesis again.
\hfill
\end{proof}

Note that, by definition, if $\sss_1 \sep \pdots \sep \sss_\pp$ is a $\Delta$-normal word, then so is every word of the form $\sss_1 \sep \pdots \sep \sss_\pp \sep 1 \sep \pdots \sep 1$: uniqueness is guaranteed only when we forbid final entries~$1$.

\subsection{Computing the $\Delta$-normal form}\label{SS:ComputingD}

Prop.\,\ref{P:NormalD} is an existential statement, which does not directly solves the question of practically computing a $\Delta$-normal decomposition of an element given by an arbitrary $\Div(\Delta)$-word. It turns out that simple incremental methods exist, which explains the interest of the $\Delta$-normal form. 

We begin with preliminary results. Their proofs are not very difficult and we give them as they are typical of what can be called the ``Garside approach''.

\begin{lemm}\label{L:Head}
Assume that $(\MM, \Delta)$ is a Garside monoid. For~$\gg$ in~$\MM$, define $\HH(\gg)$ to be the left-gcd of~$\gg$ and~$\Delta$. Then, for all $\sss_1 \wdots \sss_\pp$ in~$\Div(\Delta)$, the following conditions are equivalent:

\ITEM1 The sequence $\sss_1 \sep \pdots \sep \sss_\pp$ is $\Delta$-normal.

\ITEM2 For every~$\ii < \pp$, one has $\sss_\ii = \HH(\sss_\ii \sss_{\ii+1}\pdots \sss_\pp)$.

\ITEM3 For every~$\ii < \pp$, one has $\sss_\ii = \HH(\sss_\ii \sss_{\ii+1})$.
\end{lemm}

\begin{proof}
Assume that $\sss_1 \sep \pdots \sep \sss_\pp$ is $\Delta$
-normal. By definition, we have $\sss_\ii \dive \Delta$ and $\sss_\ii \dive \sss_\ii\sss_{\ii+} \pdots \sss_\pp$, whence $\sss_\ii \dive \HH(\sss_\ii \sss_{\ii+1}\pdots \sss_\pp)$ by the definition of a left-gcd. Conversely, let $\sss$ be the right-lcm of~$\sss_\ii$ and~$\HH(\sss_\ii \sss_{\ii+1}\pdots \sss_\pp)$. As we have $\sss_\ii \dive \Delta$ and $\HH(\sss_\ii \sss_{\ii+1}\pdots \sss_\pp) \dive \Delta$, the definition of a right-lcm implies $\sss \dive \Delta$. Hence, we have $\sss_1 \dive \sss \in \Div(\Delta)$ and $\sss \dive \sss_\ii \sss_{\ii+1}\pdots \sss_\pp$, so \eqref{E:NormalD} implies that $\sss_\ii \divs \sss$ is impossible. Therefore, we must have $\sss = \sss_\ii$, meaning that $\HH(\sss_\ii \sss_{\ii+1}\pdots \sss_\pp)$ left-divides~$\sss_\ii$. We deduce $\sss_\ii = \HH(\sss_\ii \sss_{\ii+1}\pdots \sss_\pp)$. Therefore, \ITEM1 implies~\ITEM2. 

The proof that \ITEM1 implies~\ITEM3 is exactly similar, replacing $\sss_\ii \pdots \sss_\pp$ with $\sss_\ii \sss_{\ii+1}$. 

Conversely, \ITEM2 implies $\forall\sss{\in}\Div(\Delta)\,(\sss\dive \sss_\ii \sss_{\ii+1}\pdots \sss_\pp\ \Rightarrow \sss \dive \sss_\ii)$, whence a fortiori $\forall\sss{\in}\Div(\Delta)\,(\sss\dive \sss_\ii \sss_{\ii+1}\pdots \sss_\pp\ \Rightarrow \sss_\ii \not\divs \sss)$, which is equivalent to~\eqref{E:NormalD}. So \ITEM1 and~\ITEM2 are equivalent. 

Next, by the definition of the left-gcd, $\sss_\ii \dive \HH(\sss_\ii \sss_{\ii+1}) \dive \HH(\sss_\ii \sss_{\ii+1}\pdots \sss_\pp)$ is always true, so \ITEM2 implies~\ITEM3.

Finally, let us assume~\ITEM3 and prove~\ITEM2. We give the argument for $\pp =\nobreak 3$, the general case then follows by an induction on~$\pp$. So, we assume $\sss_1 = \HH(\sss_1\sss_2)$ and $\sss_2 = \HH(\sss_2\sss_3)$, and want to prove $\sss_1 = \HH(\sss_1 \sss_2 \sss_3)$. The nontrivial argument uses the \emph{right-complement} operation. By assumption, any two elements~$\ff, \gg$ of~$\MM$ admit a right-lcm, say~$\hh$. Let us denote by~$\ff \under \gg$ and $\gg \under \ff$ the (unique) elements satisfying $\ff (\ff\under\gg) = \gg(\gg\under\ff) = \hh$. Using the associativity of the right-lcm operation, one checks that the right-complement operation~$\under$ obeys the law
\begin{equation}\label{E:RrightCompl}
(\gg\hh) \under \ff = (\gg\under\ff) \under\hh.
\end{equation}

Now, let~$\sss$ be an element of~$\Div(\Delta)$ left-dividing~$\sss_1 \sss_2 \sss_3$. Our aim is to show that $\sss$ left-divides~$\sss_1$. By assumption, the right-lcm of~$\sss$ and~$\sss_1 \sss_2 \sss_3$ is $\sss_1 \sss_2 \sss_3$, so we have $(\sss_1 \sss_2 \sss_3) \under \sss = 1$. Applying~\eqref{E:RrightCompl} with $\ff = \sss$, $\gg = \sss_1$, and $\hh = \sss_2 \sss_3$, we deduce $(\sss_1\under\sss) \under (\sss_2\sss_3) = 1$, which means that $\sss_1 \under \sss$ left-divides~$\sss_2 \sss_3$. Because $\sss_1$ and~$\sss$ divide~$\Delta$, so does their right-lcm~$\ttt$, hence so does also their right-complement~$\sss_1 \under \sss$, which is a right-divisor of~$\ttt$. So we have $\sss_1 \under \sss \dive \sss_2\sss_3$, and the assumption $\sss_2 = \HH(\sss_2 \sss_3)$ then implies $\sss_1 \under\sss \dive \sss_2$. Arguing back, we deduce that $\sss_1\sss_2$ is the right-lcm of~$\sss$ and~$\sss_1 \sss_2$, that is, that $\sss$ left-divides~$\sss_1 \sss_2$. From there, the assumption $\sss_1 = \HH(\sss_1 \sss_2)$ implies that $\sss$ left-divides~$\sss_1$, as expected. \hfill
\end{proof}

Lemma~\ref{L:Head} is not yet sufficient to compute a $\Delta$-normal decomposition for an arbitrary element, but it already implies an important property:

\begin{prop}\label{P:Regular}
If $(\MM, \Delta)$ is a Garside monoid, then $\Delta$-normal words form a regular language.
\end{prop}

\begin{proof}
By assumption, the family $\Div(\Delta)$ is finite. Moreover, by Lemma~\ref{L:Head}, a word $\sss_1 \sep \pdots \sep \sss_\pp$ is $\Delta$-normal if, and only if, each length-two factor $\sss_\ii \sep \sss_{\ii+1}$ is $\Delta$-normal. Hence, the language of all $\Delta$-normal words is defined over the alphabet $\Div(\Delta)$ by the exclusion of finitely many patterns, namely the pairs $\sss_\ii \sep \sss_{\ii+1}$ satisfying $\sss_\ii \not= \HH(\sss_\ii\sss_{\ii+1})$. Hence it is a regular language.
\end{proof}

We now specifically consider the normalisation of $\Div(\Delta)$-words of length~two.


\begin{lemm}\label{L:PsiD}
If $(\MM, \Delta)$ is a Garside monoid, then, for all~$\sss_1, \sss_2$ in~$\Div(\Delta)$, the element~$\sss_1\sss_2$ has a unique $\Delta$-normal decomposition of length~two. 
\end{lemm}

\begin{proof}
Let $\sss_1, \sss_2$ belong to~$\Div(\Delta)$. Let $\sss'_1 := \HH(\sss_1\sss_2)$. As $\MM$ is left-cancellative, there exists a unique element~$\sss'_2$ satisfying $\sss'_1 \sss'_2 = \sss_1 \sss_2$. By construction, $\sss_1$ left-divides~$\sss'_1$, which implies that $\sss'_2$ right-divides~$\sss_2$, hence a fortiori~$\Delta$. So $\sss'_2$ lies in~$\Div(\Delta)$. By construction, the $\Div(\Delta)$-word $\sss'_1 \sep \sss'_2$ is $\Delta$-normal, and it is a decomposition of~$\sss_1 \sss_2$. Three cases are possible: if $\sss'_2$ is not $1$, then $\sss'_1 \sep \sss'_2$ is a strict $\Delta$-normal decomposition of~$\sss_1 \sss_2$ (which is then said to have $\Delta$-length two); otherwise, $\sss_1\sss_2$ is a divisor of~$\Delta$, so $\sss'_1$ is a strict $\Delta$-normal decomposition of~$\sss_1 \sss_2$ (which is then said to have $\Delta$-length one), unless $\sss'_1$ is also~$1$, corresponding to $\sss_1 = \sss_2 = 1$, where the strict $\Delta$-normal decomposition is empty (and $\sss_1 \sss_2$, which is~$1$, is said to have $\Delta$-length zero).
\hfill\end{proof}

The previous result and the subsequent arguments become (much) more easily understandable when illustrated with diagrams. To this end, we associate to every element~$\gg$ of the considered monoid a $\gg$-labeled edge \begin{picture}(12,2)(0,0)\pcline{->}(1,0)(11,0)\taput{$\gg$}\end{picture}. 

\noindent\begin{minipage}{\textwidth}
\rightskip25mm
We \VR(3.2,0) then associate with a product the concatenation of the corresponding edges (which amounts to viewing the ambient monoid as a category), and represent equalities in the ambient monoid using commutative diagrams: for instance, the square on the right illustrates an equality $\ff\gg =\nobreak \ff'\gg'$.\ \hfill
\begin{picture}(0,0)(-6,-3)
\pcline{->}(0,9)(0,1)\tlput{$\ff$}
\pcline{->}(1,0)(14,0)\tbput{$\gg$}
\pcline{->}(1,10)(14,10)\taput{$\ff'$}
\pcline{->}(15,9)(15,1)\trput{$\gg'$}
\end{picture}\par
\end{minipage}

Next, \VR(3.2,0) assuming that a set of elements~$\SSS$ is given and some distinguished subset~$\FFF$ of~$\Pow\SSS2$ has been fixed, typically length-two normal words of some sort, we shall indicate that a length-two $\SSS$-word $\sss_1 \sep \sss_2$ belongs to~$\FFF$ (that is, ``is normal'') by connecting 
  
\noindent\begin{minipage}{\textwidth}
\rightskip25mm
the \VR(3.2,0) corresponding edges with a small arc, as in \begin{picture}(24,4)(0,0)\pcline{->}(1,0)(11,0)\taput{$\sss_1$}\pcline{->}(13,0)(23,0)\taput{$\sss_2$}\psarc[style=thin](11.5,0){3}{0}{180}\end{picture}. Then, \VR(3.2,0) Lemma~\ref{L:PsiD} is illustrated in the diagram on the right: it says that, for all given~$\sss_1, \sss_2$ in~$\Div(\Delta)$ (solid arrows), there exist~$\sss'_1, \sss'_2$ in~$\Div(\Delta)$ (dashed arrows) such that $\sss'_1 \sep \sss'_2$ is $\Delta$-normal and the diagram is commutative.\hfill
\begin{picture}(0,0)(-6,-4)
\pcline{->}(0,9)(0,1)\tlput{$\sss_1$}
\pcline{->}(1,0)(14,0)\tbput{$\sss_2$}
\pcline[style=exist]{->}(1,10)(14,10)\taput{$\sss'_1$}
\pcline[style=exist]{->}(15,9)(15,1)\trput{$\sss'_2$}
\psarc[style=thinexist](15,10){3.5}{180}{270}
\end{picture}\par
\end{minipage}

The \VR(3.2,0) second ingredient needed for computing the normal form involves what the call the (left) domino rule.

\begin{defi}\label{D:LeftDomino}
\rm Assume that $\MM$ is a left-cancellative monoid, $\SSS$ is a subset of~$\MM$, and $\FFF$ is a family of $\SSS$-words of length~two. We say that the \emph{left domino rule is valid for~$\FFF$} if, whenever $\sss_1, \sss_2, \sss'_1, \sss'_2, \ttt_0, \ttt_1, \ttt_2$ lie in~$\SSS$ and $\sss'_1 \ttt_1 = \ttt_0\sss_1$ and $\sss'_2 \ttt_2 = \ttt_1\sss_2$ hold in~$\MM$, then the assumption that $\sss_1 \sep \sss_2$, $\sss'_1 \sep \ttt_1$, and $\sss'_2 \sep \ttt_2$ lie in~$\FFF$ implies that $\sss'_1 \sep \sss'_2$ lies in~$\FFF$ as well.
\end{defi}

\noindent\begin{minipage}{\linewidth}
\rightskip42mm
The \VR(3,0) left domino rule corresponds to the diagram on the right: the solid arcs are the assumptions, namely that $\sss'_1 \sep \ttt_1$, $\sss'_2 \sep \ttt_2$ and $\sss_1 \sep \sss_2$ lie in~$\FFF$, the dotted arc is the expected conclusion, namely that $\sss'_1 \sep \sss'_2$ does.\hfill%
\begin{picture}(0,0)(-8,-2)
\psarc[style=thin](15,0){3}{180}{360}
\psarc[style=thin](15,10){3.5}{180}{270}
\psarc[style=thin](30,10){3.5}{180}{270}
\psarc[style=thinexist](15,10){3}{0}{180}
\pcline{->}(1,0)(14,0)\tbput{$\sss_1$}
\pcline{->}(16,0)(29,0)\tbput{$\sss_2$}
\pcline{->}(1,10)(14,10)\taput{$\sss'_1$}
\pcline{->}(16,10)(29,10)\taput{$\sss'_2$}
\pcline{->}(0,9)(0,1)\tlput{$\ttt_0$}
\pcline{->}(15,9)(15,1)\trput{$\ttt_1$}
\pcline{->}(30,9)(30,1)\trput{$\ttt_2$}
\end{picture}\par
\end{minipage}

\begin{lemm}\label{L:DominoD}
If \VR(3.2,0) $(\MM, \Delta)$ is a Garside monoid, then the left domino rule is valid for $\Delta$-normal words of length~two. 
\end{lemm}

\begin{proof}
Assume that $\sss_1, \sss_2, \sss'_1, \sss'_2, \ttt_0, \ttt_1, \ttt_2$ lie in~$\SSS$ and satisfy the assumptions of Def.\,\ref{D:LeftDomino} (with respect to $\Delta$-normal words of length~two). We want to show that $\sss'_1 \sep \sss'_2$ is $\Delta$-normal. In view of Lemma~\ref{L:Head}, assume $\sss \in \Div(\Delta)$ and $\sss \dive \sss'_1 \sss'_2$. 

\noindent\begin{minipage}{\textwidth}
\rightskip48mm
Then, \VR(3.2,0) we have $\sss \dive \sss'_1 \sss'_2 \ttt_2$, whence $\sss \dive \ttt_0 \sss_1 \sss_2$, see the diagram on the right. Arguing as in the proof of Lemma~\ref{L:Head}, we deduce $\ttt_0 \under \sss \dive \sss_1 \sss_2$. As $\ttt_0 \under\sss$  lies in~$\Div(\Delta)$, we deduce $\ttt_0 \under \sss \dive \HH(\sss_1 \sss_2) = \sss_1$, whence $\sss \dive \ttt_0 \sss_1 = \sss'_1 \ttt_1$. As $\sss$ lies in~$\Div(\Delta)$, we deduce $\sss \dive \HH(\sss'_1 \ttt_1) = \sss'_1$. Therefore, we have $\sss'_1 = \HH(\sss'_1 \sss'_2)$, and $\sss'_1 \sep \sss'_2$ is $\Delta$-normal.\hfill\qed
\begin{picture}(0,0)(-6,1.5)
\pcline{->}(11,25)(24,25)\taput{$\sss'_1$}
\pcline{->}(26,25)(39,25)\taput{$\sss'_2$}
\pcline{->}(11,10)(24,10)\tbput{$\sss_1$}
\pcline{->}(26,10)(39,10)\tbput{$\sss_2$}
\pcline{->}(10,24)(10,11)\put(11,18.5){$\ttt_0$}
\pcline{->}(25,24)(25,11)\put(26,18.5){$\ttt_1$}
\pcline{->}(40,24)(40,11)\put(41,18.5){$\ttt_2$}
\pcline{->}(9,24)(1,16)\taput{$\sss$}
\pcline{->}(9,9)(1,1)\trput{$\ttt_0\under\sss$}
\pcline{->}(0,14)(0,1)
\pcline(1,0)(25,0)\psbezier(25,0)(28,0)(30,0)(33,3)\pcline{->}(33,3)(39,9)
\pcline[border=2pt](1,15)(25,15)\psbezier[border=2pt](25,15)(28,15)(30,15)(33,18)\pcline{->}(33,18)(39,24)
\psbezier[style=exist](10,0)(13,0)(15,0)(18,3)\pcline[style=exist]{->}(18,3)(24,9)
\psbezier[style=exist](10,15)(13,15)(15,15)(18,18)\pcline[style=exist]{->}(18,18)(24,24)
\psarc[style=thin](25,10){3}{180}{360}
\psarc[style=thin](25,25){3.5}{180}{270}
\psarc[style=thin](40,25){3.5}{180}{270}
\psarc[style=thinexist](25,25){3}{0}{180}
\end{picture}
\end{minipage}
\def\qed{\relax}\end{proof}

\bigskip

We can now easily compute a $\Delta$-normal decomposition for every element. In order to describe the procedure (which can be translated into an algorithm directly), we work with $\Div(\Delta)$-words and, starting from an arbitrary $\Div(\Delta)$-word~$\ww$, explain how to build a $\Delta$-normal word that represents the same element. To this end, we first introduce notations that will be used throughout the paper.

\begin{nota}\label{N:Positions}
\rm \ITEM1 If $\SSS$ is a set and~$\FF$ is a map from~$\Pow\SSS2$ to itself, then, for $\ii\ge 1$, we denote by~$\FF_\ii$ the (partial) map of~$\SSS^*$ to itself that consists in applying~$\FF$ to the entries in position~$\ii$ and~$\ii+1$. If $\uu= \ii_1 \sep \pdots \sep \ii_{\nn}$ is a finite sequence of positive integers, we write~$\FF_{\uu}$ for the composite map $\FF_{\ii_{\nn}}\comp \pdots \comp \FF_{\ii_1}$ (so $\FF_{\ii_1}$ is applied first).

\ITEM2 If $\SSS$ is a set and $\NM$ is a map from~$\SSS^*$ to itself, we denote by~$\NMbar$ the restriction of~$\NM$ to~$\Pow\SSS2$. 
\end{nota}

Then the main result about the $\Delta$-normal form can be stated as follows:

\begin{prop}\label{P:RecipeD}
Assume that $(\MM, \Delta)$ is a Garside monoid. Then, for every $\Div(\Delta)$-word~$\ww$ of length~$\pp$, there exists a unique $\Delta$-normal word~$\NMD(\ww)$ of length~$\pp$ that represents the same element as~$\ww$. Moreover, one has
\begin{equation}\label{E:RecipeD}
\NMD(\ww) = \NMDbarr{\delta_\pp}(\ww),
\end{equation}
with $\delta_\pp$ inductively defined by $\delta_2 := 1$ and $\delta_\pp := \sh(\delta_{\pp-1}) \sep 1 \sep 2 \sep \pdots \sep \pp{-}1$, where $\sh$ is a shift of all entries by~$+1$.
\end{prop}

Thus, for instance, \eqref{E:RecipeD} says that, in order to normalise a $\Div(\Delta)$-word of length~four, we can successively normalise the length-two factors beginning at positions~$3$, $2$, $3$, $1$, $2$, and~$3$, thus in six steps.

\begin{proof}
We begin with an auxiliary result, namely finding a $\Delta$-normal decomposition for a word of the form $\ttt \sep \sss_1 \sep \pdots \sep \sss_\pp$ where $\sss_1 \sep \pdots \sep \sss_\pp$ is $\Delta$-normal, that is, for multiplying a $\Delta$-normal word by one more letter on the left. Then we claim that, putting $\ttt_0 := \ttt$ and, inductively, $\sss'_\ii \sep \ttt_\ii := \NMDbar(\ttt_{\ii-1} \sep \sss_\ii)$ for $\ii = 1 \wdots \pp$, provides a $\Delta$-normal decomposition of length~$\pp+1$ for $\ttt\sss_1 \pdots \sss_\pp$. Indeed, the commutativity of the diagram in Fig.\,\ref{F:LeftMultD} gives the equality $\ttt\sss_1 \pdots \sss_\pp = \sss'_1 \pdots \sss'_\pp\ttt_\pp$ in~$\MM$, and the validity of the left domino rule implies that each pair $\sss'_\ii \sep \sss'_{\ii+1}$ is $\Delta$-normal. In terms of~$\NMDbar$, we deduce the equality
\begin{equation}\label{E:NormalD1}
\NMD(\ttt \sep \ww) = \NMDbarr{1\sep2\sep \pdots \sep \pp}\HS{5.6}(\ttt \sep \ww).
\end{equation}
when $\ww$ is a $\Delta$-normal word of length~$\pp$. From there, \eqref{E:NormalD} follows by a straightforward induction.
\hfill\end{proof}

\begin{figure}[htb]\centering
\begin{picture}(75,15)(0,-1)
\pcline{->}(1,0)(14,0)\tbput{$\sss_1$}
\pcline{->}(16,0)(29,0)\tbput{$\sss_2$}
\pcline[style=etc](32,0)(43,0)
\pcline{->}(46,0)(59,0)\tbput{$\sss_\pp$}
\pcline{->}(1,10)(14,10)\taput{$\sss'_1$}
\pcline{->}(16,10)(29,10)\taput{$\sss'_2$}
\pcline[style=etc](32,10)(43,10)
\pcline{->}(46,10)(59,10)\taput{$\sss'_\pp$}
\pcline{->}(61,10)(74,10)\taput{$\sss'_{\pp+1}$}
\pcline{->}(0,9)(0,1)\tlput{$\ttt = \ttt_0$}
\pcline{->}(15,9)(15,1)\trput{$\ttt_1$}
\pcline{->}(30,9)(30,1)\trput{$\ttt_2$}
\pcline{->}(45,9)(45,1)\trput{$\ttt_{\pp-1}$}
\pcline{->}(60,9)(60,1)\trput{$\ttt_\pp$}
\pcline[style=double](61,0)(68,0)
\pcline[style=double](75,9)(75,7)
\psarc[style=double](68,7){7}{270}{360}
\psarc[style=thin](15,10){3.5}{180}{270}
\psarc[style=thin](30,10){3.5}{180}{270}
\psarc[style=thin](60,10){3.5}{180}{270}
\psarc[style=thin](15,0){3}{180}{360}
\psarc[style=thin](30,0){3}{180}{360}
\psarc[style=thin](45,0){3}{180}{360}
\psarc[style=thinexist](15,10){3}{0}{180}
\psarc[style=thinexist](30,10){3}{0}{180}
\psarc[style=thinexist](45,10){3}{0}{180}
\psarc[style=thin](60,10){3}{0}{180}
\end{picture}
\caption{Left-multiplying a $\Delta$-normal word by an element of~$\Div(\Delta)$: the left domino rule guarantees that the upper row is $\Delta$-normal whenever the lower row is.}
\label{F:LeftMultD}
\end{figure}
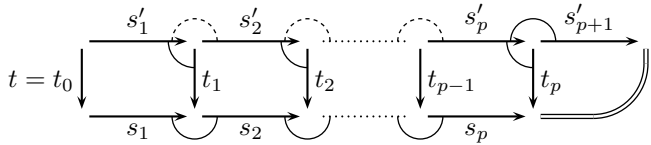

The explicit description of Prop.\,\ref{P:RecipeD} enables one to completely analyse the complexity of the $\Delta$-normal form.

\begin{coro}\label{C:ComplexD}
If $(\MM, \Delta)$ is a Garside monoid, then $\Delta$-normal decompositions can be computed in linear space and quadratic time. The Word Problem for~$\MM$ with respect to~$\Div(\Delta)$ lies in~$\textsc{dtime}(\nn^2)$.
\end{coro}

\begin{proof}
By assumption, the set $\Div(\Delta)$ is finite, so the complete table of the map~$\NMDbar$ can be precomputed, and then each application of~$\NMDbar$ has time cost~$O(1)$ and keeps the length unchanged. Then, as the sequence~$\delta_\pp$ has length $\pp(\pp-1)/2$, Prop.\,\ref{P:RecipeD} implies that a $\Delta$-normal decomposition for an element represented by a $\Div(\Delta)$-word of length~$\pp$ can be obtained in time~$O(\pp^2)$.

Computing $\Delta$-normal decompositions yields a solution of the Word Problem, since two $\Div(\Delta)$-words~$\ww, \ww'$ represent the same element of~$\MM$ if, and only if, the $\Delta$-normal words~$\NMD(\ww)$ and~$\NMD(\ww')$ only differ by the possible adjunction of final letters~$1$.
\end{proof}

On the other hand, a direct application of~\eqref{E:NormalD1} and Fig.\,\ref{F:LeftMultD} is the fact that, viewed as paths in the Cayley graph of~$\MM$ with respect to~$\Div(\Delta)$, the $\Delta$-normal forms of~$\gg$ and~$\ttt\gg$ remain at a uniformly bounded distance, namely at most two. Thus, we can state (see~\cite{HoTh}):

\begin{coro}\label{C:LeftFTPD}
If $(\MM, \Delta)$ is a Garside monoid, then the $\Delta$-normal words satisfiy the $2$-Fellow traveller Property on the left.
\end{coro}

\subsection{The right counterpart}\label{SS:RightMultD}

Owing to Prop.\,\ref{P:RecipeD} and its normalisation recipe based on left-multiplication, the question naturally arises of a similar recipe based on right-multiplication, hence based on computing a $\Delta$-normal decomposition of~$\gg\ttt$ from one of~$\gg$. Such a recipe does exist, but this is not obvious, because the definition of $\Delta$-normality is not symmetric.

\begin{prop}\label{P:RecipeRD}
Assume that $(\MM, \Delta)$ is a Garside monoid. Then, for every $\Div(\Delta)$-word~$\ww$ of length~$\pp$, one also has
\begin{equation}\label{E:RecipeRD}
\NMD(\ww) = \NMDbarr{\widetilde\delta_\pp}(\ww),\VR(4,1)
\end{equation}
with $\widetilde\delta_\pp$ inductively defined by $\widetilde\delta_2 := 1$ and $\widetilde\delta_\pp := \widetilde\delta_{\pp-1} \sep \pp{-}1 \sep \pdots \sep 2 \sep 1$.
\end{prop}

For instance, \eqref{E:RecipeRD} says that, in order to normalise a $\Div(\Delta)$-word of length four, we can successively normalise the length-two factors beginning at positions~$1$, $2$, $1$, $3$, $2$, and~$1$, in six steps as in~\eqref{E:RecipeD}, but in a different order.

As can be expected, the proof of Prop.\,\ref{P:RecipeRD} relies on a symmetric counterpart of the left domino rule of Def.\,\ref{D:LeftDomino}.

\begin{defi}\label{D:RightDomino}
\rm Assume that $\MM$ is a left-cancellative monoid, $\SSS$ is a subset of~$\MM$, and $\FFF$ is a family of $\SSS$-words of length~two. We say that the \emph{right domino rule is valid for~$\FFF$} if, whenever $\sss_1, \sss_2, \sss'_1, \sss'_2, \ttt_0, \ttt_1, \ttt_2$ lie in~$\SSS$ and $\ttt_0\sss'_1 = \sss_1 \ttt_1$ and $\ttt_1\sss'_2 = \sss_2 \ttt_2$ hold in~$\MM$, then the assumption that $\sss_1 \sep \sss_2$, $\ttt_0 \sep \sss'_1$, and $\ttt_1 \sep \sss'_2$ lie in~$\FFF$ implies that $\sss'_1 \sep \sss'_2$ lies in~$\FFF$ as well.
\end{defi}

\noindent\begin{minipage}{\linewidth}
\rightskip42mm
The \VR(3.2,0) right domino rule corresponds to the diagram on the right: the solid arcs are the assumptions, namely that $\ttt_0 \sep \sss'_1$, $\ttt_1 \sep \sss'_2$ and $\sss_1 \sep \sss_2$ lie in~$\FFF$, and the dotted arc is the expected conclusion, namely that $\sss'_1 \sep \sss'_2$ does.\hfill%
\begin{picture}(0,0)(-7,-2.5)
\psarc[style=thinexist](15,0){3}{180}{360}
\psarc[style=thin](0,0){3.5}{0}{90}
\psarc[style=thin](15,0){3.5}{0}{90}
\psarc[style=thin](15,10){3}{0}{180}
\pcline{->}(1,0)(14,0)\tbput{$\sss'_1$}
\pcline{->}(16,0)(29,0)\tbput{$\sss'_2$}
\pcline{->}(1,10)(14,10)\taput{$\sss_1$}
\pcline{->}(16,10)(29,10)\taput{$\sss_2$}
\pcline{->}(0,9)(0,1)\tlput{$\ttt_0$}
\pcline{->}(15,9)(15,1)\trput{$\ttt_1$}
\pcline{->}(30,9)(30,1)\trput{$\ttt_2$}
\end{picture}\par
\end{minipage}

Then \VR(3.2,0) we have the counterpart of Lemma~\ref{L:DominoD}. Observe that the argument is totally different, reflecting the lack of symmetry in the definition of $\Delta$-normality.

\pagebreak

\begin{lemm}\label{L:DominoRD}
If $(\MM, \Delta)$ is a Garside monoid, then the right domino rule is valid for $\Delta$-normal words of length~two.
\end{lemm}

\begin{proof}[Proof (sketch)]
For~$\sss$ in~$\Div(\Delta)$, let $\partial\sss$ be the element of~$\Div(\Delta)$ satisfying $\sss \partial\sss = \Delta$, and let $\phi:= \partial^2$. Then, $\sss \Delta = \Delta \phi(\sss)$ holds for every~$\sss$ in~$\Div(\Delta)$, and one shows that $\phi$ extends to an automorphism of~$\MM$. It follows, in particular, that $\sss_1 \sep \sss_2$ being $\Delta$-normal implies that $\phi(\sss_1) \sep \phi(\sss_2)$ is $\Delta$-normal as well. By assumption, there exist $\ttt'_0, \ttt'_1, \ttt'_2$ satisfying $\ttt_0 \ttt'_0 = \ttt_1 \ttt'_1 = \ttt_2 \ttt_2' = \Delta$. By the above equality, we have $\sss_1 \Delta = \Delta \phi(\sss_1)$ and $\sss_2 \Delta = \Delta \phi(\sss_2)$, whence, by left-cancellation, $\sss'_1 \ttt'_1 = \ttt'_0 \phi(\sss_1)$ and $\sss'_2 \ttt'_2 = \ttt'_1 \phi(\sss_1)$. Thus the diagram below is commutative. Assume $\sss \dive \sss'_1 \sss'_2$ with $\sss$ in~$\Div(\Delta)$. Let $\sss' :=\nobreak \sss'_1 \under \sss$. Our aim is to prove that $\sss$ left-divides~$\sss'_1$, that is, that $\sss'$ is~$1$. 

\noindent\begin{minipage}{\linewidth}  
\rightskip42mm
The \VR(3.2,0) assumption that $\sss$ left-divides~$\sss'_1 \sss'_2$ implies $\sss' \dive \sss'_2$, whence $\ttt_1\sss' \dive \ttt_1 \sss'_2$. On the other hand, the assumption $\sss \dive \sss'_1 \sss'_2$ implies a fortiori $\sss \dive \sss'_1 \sss'_2 \ttt'_2$, that is, $\sss \dive \ttt'_0 \phi(\sss_1) \phi(\sss_2)$ and, therefore, $\ttt'_0 \under \sss \dive \phi(\sss_1) \phi(\sss_2)$. As $\ttt'_0 \under \sss$ lies in~$\Div(\Delta)$ and $\phi(\sss_1) \sep \phi(\sss_2)$ is $\Delta$-normal, we deduce $\ttt'_0 \under \sss \dive \phi(\sss_1) \phi(\sss_2)$, whence $\sss \dive \ttt'_0 \phi(\sss_1)$, which is also%
\hfill\begin{picture}(0,0)(-7,-5)
\psarc[style=thinexist](15,10){3}{180}{360}
\psarc[style=thin](0,10){3.5}{0}{90}
\psarc[style=thin](15,10){3.5}{0}{90}
\psarc[style=thin](15,20){3}{0}{180}
\pcline{->}(1,0)(14,0)\tbput{$\phi(\sss_1)$}
\pcline{->}(16,0)(29,0)\tbput{$\phi(\sss_2)$}
\pcline{->}(1,10)(14,10)\taput{$\sss'_1$}
\pcline{->}(16,10)(29,10)\taput{$\sss'_2$}
\pcline{->}(1,20)(14,20)\taput{$\sss_1$}
\pcline{->}(16,20)(29,20)\taput{$\sss_2$}
\pcline{->}(0,9)(0,1)\tlput{$\ttt'_0$}
\pcline{->}(15,9)(15,1)\trput{$\ttt'_1$}
\pcline{->}(30,9)(30,1)\trput{$\ttt'_2$}
\pcline{->}(0,19)(0,11)\tlput{$\ttt_0$}
\pcline{->}(15,19)(15,11)\trput{$\ttt_1$}
\pcline{->}(30,19)(30,11)\trput{$\ttt_2$}
\end{picture}\par
\end{minipage}
$\sss \dive \sss'_1 \ttt'_1$. We \VR(3.2,0) deduce $\sss' \dive \ttt'_1$, and, therefore, $\ttt_1 \sss' \dive \ttt_1 \ttt'_1 = \Delta$. Thus $\ttt_1 \sss'$ lies in~$\Div(\Delta)$ and it left-divides~$\ttt_1 \sss'_2$. By assumption, $\ttt_1 \sep \sss'_2$ is $\Delta$-normal, so we deduce that $\ttt_1 \sss'$ left-divides~$\ttt_1$, implying $\sss' = 1$, as expected. Hence, $\sss'_1 \sep \sss'_2$ is $\Delta$-normal, the right domino rule is valid.\hfill
\end{proof}

We can easily complete the argument.

\begin{proof}[Proof of Prop.\,\ref{P:RecipeRD}]
The argument is symmetric of the one for Prop.\,\ref{P:RecipeD}. It consists in establishing for $\ww$ a $\Delta$-normal word of length~$\pp$ the equality
\begin{equation}\label{E:NormalRD1}
\NMD(\ww \sep \ttt) = \NMDbarr{\pp\sep\pp-1\sep \pdots \sep 1}\HS{8}(\ww).
\end{equation}
The latter immediately follows from the diagram
$$\begin{picture}(75,16)(-15,-2)
\pcline{->}(-14,0)(-1,0)\tbput{$\sss'_0$}
\pcline{->}(1,0)(14,0)\tbput{$\sss'_1$}
\pcline{->}(31,0)(44,0)\tbput{$\sss'_{\pp-1}$}
\pcline[style=etc](17,0)(28,0)
\pcline{->}(46,0)(59,0)\tbput{$\sss'_\pp$}
\pcline{->}(1,10)(14,10)\taput{$\sss_1$}
\pcline{->}(31,10)(44,10)\taput{$\sss_{\pp-1}$}
\pcline[style=etc](17,10)(28,10)
\pcline{->}(46,10)(59,10)\taput{$\sss_\pp$}
\pcline{->}(0,9)(0,1)\put(-3.5,5){$\ttt_0$}
\pcline{->}(15,9)(15,1)\put(16,5){$\ttt_1$}
\pcline{->}(30,9)(30,1)\put(31,5){$\ttt_2$}
\pcline{->}(45,9)(45,1)\put(46,5){$\ttt_{\pp-1}$}
\pcline{->}(60,9)(60,1)\put(61,5){$\ttt_\pp = \ttt$}
\pcline[style=double](-1,10)(-8,10)
\pcline[style=double](-15,3)(-15,1)
\psarc[style=double](-8,3){7}{90}{180}
\psarc[style=thin](30,0){3.5}{0}{90}
\psarc[style=thin](45,0){3.5}{0}{90}
\psarc[style=thin](0,0){3.5}{0}{90}
\psarc[style=thinexist](0,0){3}{180}{360}
\psarc[style=thinexist](15,0){3}{180}{360}
\psarc[style=thinexist](30,0){3}{180}{360}
\psarc[style=thinexist](45,0){3}{180}{360}
\psarc[style=thin](15,10){3}{0}{180}
\psarc[style=thin](30,10){3}{0}{180}
\psarc[style=thin](45,10){3}{0}{180}
\end{picture}$$
whose validity is guaranteed by the right domino rule.
\hfill\end{proof}

We deduce a counterpart of Cor.\,\ref{C:LeftFTPD}:

\begin{coro}\label{C:RightFTPD}
Assume that $(\MM, \Delta)$ is a Garside monoid. Then $\Delta$-normal words satisfy the $2$-Fellow traveller Property on the right.
\end{coro}

\begin{rema}
\rm As in Subsections~\ref{SS:Abelian} and~\ref{SS:Braids}, the $\Delta$-normal decompositions associated with a Garside monoid~$(\MM, \Delta)$ extend to the group of fractions of~$\MM$. It directly follows from the definition that $\MM$ satisfies the Ore conditions (cancellativity and existence of common right-multiples), hence embeds in a group of (left) fractions~$\GG$ (then called a \emph{Garside group}). Then every element of~$\GG$ admits a unique decomposition of the form $\Delta^\mm \sep \sss_1 \sep \pdots \sep \sss_\pp$ where $\sss_1 \sep \pdots \sep \sss_\pp$ is $\Delta$-normal and, in addition, we require $\sss_1 \not= \Delta$ (that is, $\mm$ is maximal) and $\sss_\pp \not= 1$ (that is, $\pp$ is minimal). It is easy to deduce from Prop.\,\ref{P:Regular} and Cor.\,\ref{C:LeftFTPD} and~\ref{C:RightFTPD} that the $\Delta$-normal form provides a biautomatic structure for~$\GG$ (in the sense of~\cite{Eps}).
\end{rema}

\section{The $\SSS$-normal form associated with a Garside family}\label{S:NormalS}

Looking at the mechanism of the $\Delta$-normal form associated with a Garside monoid invites to a further extension. Indeed, one quickly sees that several assumptions in the definition of a Garside monoid are not used in the construction of a normal form that obeys the recipe of Prop.\,\ref{P:RecipeD}. This easy observation, and the need of using decompositions similar to $\Delta$-normal ones in more general situations, led to introducing the notion of a \emph{Garside family}~\cite{Dif}, extensively developed in the book~\cite{Garside}. Also see~\cite{Dig} for the computational aspects.

\subsection{The notion of a Garside family}\label{SS:Greedy}

Our aim is to define normal forms that work in the same way as the $\Delta$-normal form of a Garside monoid, but in more general monoids (in fact, monoids can be extended into categories at no cost). So, we start with a monoid~$\MM$ equipped with a generating family~$\SSS$ and try to define for the elements of~$\MM$ distinguished $\SSS$-decompositions that resemble $\Delta$-normal decompositions: in particular, if $(\MM, \Delta)$ is a Garside monoid and $\SSS$ is $\Div(\Delta)$, we should retrieve $\Delta$-normal decompositions. Of course, we cannot expect to do that for an arbitrary generating family~$\SSS$, and this is where the notion of a Garside family will appear. First, if we try to just copy~\eqref{E:NormalD}, problems arise. Therefore, we start from a new notion.

\begin{defi}\label{D:NormalS}
\rm If $\MM$ is a left-cancellative monoid and~$\SSS$ is included in~$\MM$, an $\SSS$-word~$\sss_1\sep\sss_2$ is called \emph{$\SSS$-normal} if the following condition holds:
\begin{equation}\label{E:NormalS}
\forall\sss{\in}\SSS\ \forall\ff{\in}\MM\: (\sss \dive \ff \sss_1\sss_2 \Rightarrow \sss\dive \ff \sss_1).
\end{equation}
An $\SSS$-word $\sss_1 \sep \pdots \sep \sss_\pp$ is called \emph{$\SSS$-normal} if $\sss_\ii \sep \sss_{\ii+1}$ is $\SSS$-normal for every~$\ii < \pp$, and \emph{strict $\SSS$-normal} if it is $\SSS$-normal with, in addition, $\sss_\pp \not= 1$.
\end{defi}

Relation~\eqref{E:NormalS} is reminiscent of~\ITEM3 in Lemma~\ref{L:Head}, but with the important difference of the additional term~$\ff$: we do not only consider the left-divisors of~$\sss_1 \sss_2$ that lie in~$\SSS$, but, more generally, all elements of~$\SSS$ that left-divide~$\ff \sss_1 \sss_2$.

\begin{exam}\label{X:Connection}
\rm Assume that $(\MM, \Delta)$ is a Garside monoid, and let $\SSS:= \Div(\Delta)$. Assume that $\sss_1 \sep \pdots \sep \sss_\pp$ is $\SSS$-normal in the sense of Def.\,\ref{D:NormalS}. Then, for every~$\ii$, \eqref{E:NormalS} implies in particular $\forall\sss{\in}\SSS\: (\sss \dive \sss_1\sss_2 \Rightarrow \sss\dive \sss_1)$, whence $\sss_\ii = \HH(\sss_\ii \sss_{\ii+1})$. Hence, by Lemma~\ref{L:Head}, $\sss_1 \sep \pdots \sep \sss_\pp$ is $\Delta$-normal in the sense of Prop.\,\ref{P:NormalD}. 

 The converse implication is also true, but less obvious. Indeed, assume that $\sss_1 \sep \pdots \sep \sss_\pp$ is $\Delta$-normal, and we have $\sss \dive \ff \sss_\ii \sss_{\ii+1}$ for some~$\sss$ in~$\Div(\Delta)$ and~$\ff$ in~$\MM$. Then, using the right-complement operation~$\under$ as in the proof of Lemma~\ref{L:Head}, we deduce that $\ff \under\sss$ left-divides~$\sss_\ii \sss_{\ii+1}$, as illustrated in the diagram
$$\begin{picture}(60,14)(0,0)
\pcline{->}(1,0)(29,0)
\pcline(31,0)(53,0)\psarc(53,7){7}{270}{360}\pcline{->}(60,7)(60,9)
\pcline{->}(1,10)(29,10)\taput{$\ff$}
\pcline{->}(31,10)(44,10)\taput{$\sss_1$}
\pcline{->}(46,10)(59,10)\taput{$\sss_2$}
\pcline{->}(0,9)(0,1)\tlput{$\sss$}
\pcline{->}(30,9)(30,1)\trput{$\ff\under\sss$}
\psarc[style=thin](45,10 ){3}{0}{180}
\psarc[style=exist](38,7){7}{270}{360}\pcline[style=exist]{->}(45,7)(45,9)
\end{picture}$$
As $\sss_\ii \sep \sss_{\ii+1}$ is $\Delta$-normal, we deduce $\ff \under\sss \dive \sss_\ii$, whence $\sss \dive \ff \sss_\ii$. Hence \eqref{E:NormalS} is satisfied and $\sss_\ii \sep \sss_{\ii+1}$ is $\SSS$-normal in the sense of Def.\,\ref{D:NormalS}.
\end{exam}

By definition, being $\SSS$-normal is a local property only involving length-two subfactors. So we immediately obtain:

\begin{prop}\label{P:RegularS}
If $\MM$ is a monoid and $\SSS$ is a finite subfamily of~$\MM$, then $\SSS$-normal words form a regular language.
\end{prop}

Hereafter we investigate $\SSS$-normal decompositions. An easy, but important fact is that such decompositions are necessarily (almost) unique when they exist. We shall restrict to the case of monoids that admit no nontrivial invertible element (as Garside monoids do). This restriction is not necessary, but it makes statements more simple: essentially, one can cope with nontrivial invertible elements at the expense of replacing equality with the weaker equivalence relation~$=^{\scriptscriptstyle\times}$, where $\gg=^{\scriptscriptstyle\times} \gg'$ means $\gg = \gg\ee$ for some invertible element~$\ee$, see~\cite{Garside}. 

\begin{lemm}\label{L:Uniqueness}
Assume that~$\MM$ is a left-cancellative monoid with no nontrivial invertible elements and~$\SSS$ is included in~$\MM$. Then every element of~$\gg$ admits at most one strict $\SSS$-normal decomposition.
\end{lemm}

\begin{proof}[sketch]
Assume that $\sss_1 \sep \pdots \sep \sss_\pp$ and $\ttt_1 \sep \pdots \sep \ttt_\qq$ are $\SSS$-normal decompositions of an element~$\gg$. From the assumption that $\sss_1$ lies in~$\SSS$ and left-divides~$\ttt_1 \pdots \ttt_\qq$, and that $\ttt_1 \sep \pdots \sep \ttt_\qq$ is $\SSS$-normal, one easily deduces $\sss_1 \dive \ttt_1$. By a symmetric argument, one deduces $\ttt_1 \dive \sss_1$, whence $\sss_1 = \ttt_1$, because $\MM$ has non nontrivial invertible element. Then use an induction.
\hfill\end{proof}

If we consider $\SSS$-normal decompositions that are not strict, uniqueness is no longer true as, trivially, $\sss_1 \sep \pdots \sep \sss_\pp$ being $\SSS$-normal implies that $\sss_1 \sep \pdots \sep \sss_\pp \sep 1 \sep \pdots \sep 1$ is also $\SSS$-normal (and represents the same element of the ambient monoid). Lemma~\ref{L:Uniqueness} says that this is the only lack of uniqueness.

At this point, we are left with the question of the existence of $\SSS$-normal decompositions, and this is where the central technical notion arises:

\begin{defi}\label{D:GarNormal}
\rm Assume that $\MM$ is a left-cancellative monoid with no nontrivial invertible elements and $\SSS$ is a subset of~$\MM$ that contains~$1$. We say that $\SSS$ is a \emph{Garside family in~$\MM$} if every element~$\gg$ of~$\MM$ has an $\SSS$-normal decomposition, that is, there exists an $\SSS$-normal $\SSS$-word $\sss_1 \sep \pdots \sep \sss_\pp$ satisfying $\sss_1 \pdots \sss_\pp = \gg$.
\end{defi}

\begin{exam}\label{X:GarFam}
\rm It follows from the connection of Ex.\,\ref{X:Connection} that, if $(\MM, \Div(\Delta)$ is a Garside monoid, then $\Div(\Delta)$ is a  Garside family in~$\MM$. So, in particular, the $\nn$-cube~$\SSS_\nn$ of~\eqref{E:AbelianCube} is a Garside family in the abelian monoid~$\NNNN^\nn$ and, similarly, the family of all simple $\nn$-strand braids is a Garside family in the braid monoid~$\BP\nn$.

Many examples of a different flavour exist. For instance, \emph{every} left-cancellative monoid~$\MM$ is a Garside family in itself, since every element~$\gg$ of~$\MM$ admits the length-one $\MM$-normal decomposition~$\gg$ (!). More interestingly, let $K^+$ be the ``Klein bottle monoid''
$$K^+ := \MON{\tta, \ttb}{\tta = \ttb\tta\ttb},$$
which is the positive cone in the ordered group $\GR{\tta, \ttb}{\tta = \ttb\tta\ttb}$, itself the fundamental group of the Klein bottle, and the nontrivial semidirect product $\ZZZZ \rtimes \ZZZZ$. Then $K^+$ cannot be made a Garside monoid since no function~$\lambda$ as in Def.\,\ref{D:GarMon}\ITEM1 may exist. However, if we put $\Delta := \tta^2$, the left- and right-divisors of~$\Delta$ coincide and the family $\Div(\Delta)$ is an (infinite)  Garside family in~$K^+$. 

We refer to~\cite{Garside} for many examples of Garside families, and just mention the recent result~\cite{Din} that every finitely generated Artin--Tits monoid admits a \emph{finite} Garside family, independently of whether the associated Coxeter group is finite or not. In Fig.\,\ref{F:Artin}, we display such a finite Garside family for the monoid with presentation 
$$\MON{\sig1, \sig2, \sig3}{\sig1\sig2\sig1 = \sig2\sig1\sig2, \sig2\sig3\sig2 = \sig3\sig2\sig3, \sig3\sig1\sig3 = \sig1\sig3\sig1},$$
that is, for what is called type~$\Att$. 
\end{exam}

\begin{figure}[htb]\centering
\begin{picture}(73,35)(-13,1)
\psset{xunit=0.9mm}\psset{yunit=0.55mm}
\psset{fillstyle=solid,fillcolor=black}
\cnode(20,30){0.5}{e}
\cnode(30,30){0.5}{a}
\cnode(15,40){0.5}{b}
\cnode(15,20){0.5}{c}
\cnode(20,50){0.5}{ba}
\cnode(35,40){0.5}{ab}
\cnode(5,20){0.5}{cb}
\cnode(5,40){0.5}{bc}
\cnode(20,10){0.5}{ca}
\cnode(35,20){0.5}{ac}
\cnode(30,50){0.5}{aba}
\cnode(0,30){0.5}{bcb}
\cnode(30,10){0.5}{cac}
\cnode[fillcolor=white,linecolor=gray](45,40){0.5}{abc}
\cnode[fillcolor=white,linecolor=gray](45,20){0.5}{acb}
\cnode[fillcolor=white,linecolor=gray](0,50){0.5}{bca}
\cnode[fillcolor=white,linecolor=gray](15,60){0.5}{bac}
\cnode[fillcolor=white,linecolor=gray](15,0){0.5}{cab}
\cnode[fillcolor=white,linecolor=gray](0,10){0.5}{cba}
\cnode(5,0){0.5}{caba}
\cnode(50,30){0.5}{abcb}
\cnode(5,60){0.5}{bcac}
\psset{nodesep=1mm}
\AArrow(e,a)
\BArrow(e,b)
\CArrow(e,c)
\BArrow(a,ab)
\CArrow(a,ac)
\AArrow(b,ba)
\CArrow(b,bc)
\AArrow(c,ca)
\BArrow(c,cb)
\AArrow(ab,aba)
\CArrow(ab,abc)
\BArrow(ba,aba)
\CArrow(ba,bac)
\AArrow(bc,bca)
\BArrow(bc,bcb)
\AArrow(cb,cba)
\CArrow(cb,bcb)
\AArrow(ac,cac)
\BArrow(ac,acb)
\CArrow(ca,cac)
\BArrow(ca,cab)
\BArrow(abc,abcb)
\CArrow(acb,abcb)
\CArrow(bca,bcac)
\AArrow(bac,bcac)
\AArrow(cab,caba)
\BArrow(cba,caba)
\rput[r](e){$1$\ \ }
\rput[l](a){\ \ $\sig1$}
\rput[l](b){\ \ $\sig2$}
\rput[l](c){\ \ $\sig3$}
\rput[r](ab){$\sig1\sig2$\ \ }
\rput[r](ba){$\sig2\sig1$\ \ }
\rput[r](bc){$\sig2\sig3$\ \ }
\rput[r](cb){$\sig3\sig2$\ \ }
\rput[r](ac){$\sig1\sig3$\ \ }
\rput[r](ca){$\sig3\sig1$\ \ }
\rput[r](bcb){$\sig2\sig3\sig2\ \ $}
\rput[l](bac){\ \ $\color{gray}\scriptstyle(\sig2\sig1\sig3)$\ }
\rput[r](bca){$\color{gray}\scriptstyle(\sig2\sig3\sig1)$\ \ }
\rput[l](cac){\ \ $\sig3\sig1\sig3$}
\rput[l](cab){\ \ $\color{gray}\scriptstyle(\sig3\sig1\sig2)$}
\rput[r](cba){$\color{gray}\scriptstyle(\sig3\sig2\sig1)$\ \ }
\rput[l](aba){\ \ $\sig1\sig2\sig1$}
\rput[l](abc){\ \ $\color{gray}\scriptstyle(\sig1\sig2\sig3)$}
\rput[l](acb){\ \ $\color{gray}\scriptstyle(\sig1\sig3\sig2)$}
\rput[l](abcb){\ \ $\sig1\sig2\sig3\sig2$}
\rput[r](bcac){$\sig2\sig3\sig1\sig3$\ \ }
\rput[r](caba){\ $\sig3\sig1\sig2\sig1$\ \ }
\end{picture}
\caption{A finite Garside family~$\SSS$ in the Artin--Tits monoid of type~$\Att$: the sixteen right-divisors of the elements $\sig1\sig2\sig3\sig2$, $\sig2\sig3\sig1\sig3$, and $\sig3\sig1\sig3\sig1$. Attention! The family~$\SSS$ is not closed under left-divisor, implying that some intermediate vertices (the six grey ones) do not belong to~$\SSS$.}
\label{F:Artin}
\end{figure}

\subsection{Computing $\SSS$-normal decompositions}\label{SS:GarNF}

We postpone to the next subsection the question of recognising Garside families, and explain here how $\SSS$-normal decompositions behave when they exist, that is, when $\SSS$ is a Garside family. To this end, the point is that the counterparts of Lemmas~\ref{L:PsiD} and~\ref{L:DominoD} are valid.

\begin{lemm}\label{L:PsiS}
Assume that $\MM$ is a left-cancellative monoid with no nontrivial invertible element and $\SSS$ is a Garside family in~$\MM$. Then, for all~$\sss_1, \sss_2$ in~$\SSS$, the element~$\sss_1\sss_2$ has a unique $\SSS$-normal decomposition of length~two. 
\end{lemm}

\begin{proof}
Let $\sss_1, \sss_2$ belong to~$\SSS$. By assumption, $\sss_1 \sss_2$ admits an $\SSS$-normal decomposition, say $\sss'_1 \sep \pdots \sep \sss'_\pp$. As $\sss_1$ belongs to~$\SSS$ and $\sss'_{\pp-1} \sep \sss'_\pp$ is $\SSS$-normal, the assumption $\sss_1 \dive  (\sss'_1 \pdots \sss'_{\pp-2}) \sss'_{\pp-1} \sss'_\pp$ implies $\sss_1 \dive (\sss'_1 \pdots \sss'_{\pp-2}) \sss'_{\pp-1}$. Repeating the argument $\pp-1$ times, we conclude that $\sss_1$ left-divides~$\sss'_1$, say $\sss'_1 = \sss_1 \ttt_1$. Left-cancelling~$\sss_1$, we deduce $\sss_2 = \ttt_1 \sss'_2 \pdots \sss'_\pp$ and, arguing as above, we conclude that $\sss_2$ must left-divide~$\ttt_1 \sss'_2$, say $\ttt_1 \sss'_2 = \sss_2 \ttt_2$. Left-cancelling~$\sss_2$, we deduce $1 = \ttt_2 \sss'_3 \pdots \sss'_\pp$. As $\MM$ has no nontrivial invertible element, the only possibility is $\ttt_1 = \sss'_3 = \pdots  = \sss'_\pp = 1$, which implies that $\sss'_1 \sep \sss'_2$ is an $\SSS$-normal decomposition of~$\sss_1\sss_2$. The argument is illustrated in the diagram
$$\begin{picture}(75,16)(0,-3)
\pcline{->}(1,10)(14,10)\taput{$\sss_1$}
\pcline{->}(16,10)(29,10)\taput{$\sss_2$}
\pcline{->}(1,0)(14,0)\tbput{$\sss'_1$}
\pcline{->}(16,0)(29,0)\tbput{$\sss'_2$}
\pcline{->}(31,0)(44,0)\tbput{$\sss'_3$}
\pcline[style=etc](47,0)(58,0)
\pcline{->}(61,0)(74,0)\tbput{$\sss'_\pp$}
\pcline[style=double](0,9)(0,1)
\pcline[style=double](31,10)(68,10)
\psarc[style=double](68,3){7}{0}{90}
\pcline[style=double](75,3)(75,1)
\pcline[style=exist]{->}(15,9)(15,1)\trput{$\ttt_1$}
\pcline[style=exist]{->}(30,9)(30,1)\trput{$\ttt_2$}
\psarc[style=thin](14.5,0){3}{180}{360}
\psarc[style=thin](29.5,0){3}{180}{360}
\psarc[style=thin](44.5,0){3}{180}{360}
\psarc[style=thin](59.5,0){3}{180}{360}
\end{picture}$$
Thus, every element of~$\Pow\SSS2$ has an $\SSS$-normal decomposition of length two. Its uniqueness follows from Lemma~\ref{L:Uniqueness}.
\end{proof}

\pagebreak

\begin{lemm}\label{L:DominoS}
Assume that $\MM$ is a left-cancellative monoid with no nontrivial invertible element and $\SSS$ is a Garside family in~$\MM$. Then the left domino rule is valid for $\SSS$-normal words of length~two.
\end{lemm}

\begin{proof}
Assume that $\sss_1, \sss_2, \sss'_1, \sss'_2, \ttt_0, \ttt_1, \ttt_2$ lie in~$\SSS$ and satisfy the assumptions of Def.\,\ref{D:LeftDomino} (with respect to $\SSS$-normal words of length 

\noindent\begin{minipage}{\textwidth}
\rightskip 43 mm
two). Our aim is \VR(3.2,0) to show that $\sss'_1 \sep \sss'_2$ is $\SSS$-normal. Assume $\sss \in \SSS$ and $\sss \dive \ff \sss'_1 \sss'_2$. First, $\sss \dive \sss'_1 \sss'_2$ implies $\sss \dive \ff \sss'_1 \sss'_2 \ttt_2$, whence $\sss \dive \ff \ttt_0 \sss_1 \sss_2$. As $\sss_1 \sep \sss_2$ is $\SSS$-normal, we deduce $\sss \dive \ff \ttt_0 \sss_1$, whence $\sss \dive \ff \sss'_1 \ttt_1$. As $\sss'_1 \sep \ttt_1$ is $\SSS$-normal, we deduce $\sss \dive \ff \sss'_1$ in turn. Therefore, $\sss'_1 \sep \sss'_2$ is $\SSS$-normal.\hfill\qed%
\begin{picture}(0,0)(-8,-2)
\pcline{->}(1,0)(14,0)\tbput{$\sss_1$}
\pcline{->}(16,0)(29,0)\tbput{$\sss_2$}
\pcline{->}(1,10)(14,10)\taput{$\sss'_1$}
\pcline{->}(16,10)(29,10)\taput{$\sss'_2$}
\pcline{->}(0,9)(0,1)\tlput{$\ttt_0$}
\pcline{->}(0,19)(0,9)\tlput{$\ff$}
\pcline{->}(15,9)(15,1)\tlput{$\ttt_1$}
\pcline{->}(30,9)(30,1)\trput{$\ttt_2$}
\psarc[style=thin](14.5,0){3}{180}{360}
\psarc[style=thin](15,10){3.5}{180}{270}
\psarc[style=thin](30,10){3.5}{180}{270}
\psarc[style=thinexist](14.5,10){3}{0}{180}
\pcline{->}(1,20)(14,20)\taput{$\sss$}
\psline(16,20)(23,20)\psarc(23,13){7}{0}{90}
\psline{->}(30,13)(30,11)
\pcline[style=exist]{->}(15,19)(15,11)
\psbezier[style=exist,border=3pt]{->}(16,19)(20,13)(20,7)(16,1)
\end{picture}
\end{minipage}
\def\qed{\relax}\end{proof}

Arguing exactly as for Prop.\,\ref{P:RecipeD} and using, in particular, Fig.\,\ref{F:LeftMultD}, we obtain

\begin{prop}\label{P:RecipeS}
Assume that $\MM$ is a left-cancellative monoid with no nontrivial invertible element and $\SSS$ is a Garside family in~$\MM$. Then, for every $\SSS$-word~$\ww$ of length~$\pp$, there exists a unique $\SSS$-normal word~$\NMS(\ww)$ of length~$\pp$ that represents the same element as~$\ww$. Moreover, with $\delta_\pp$ as in Prop.\,\ref{P:RecipeD}, one has
\begin{equation}\label{E:RecipeS}
\NMS(\ww) = \NMSbarr{\delta_\pp}(\ww).
\end{equation}
\end{prop}

Thus, the recipe for computing the $\Delta$-normal form associated with a Garside monoid extends without change to the $\SSS$-normal form associated with an arbitrary Garside family~$\SSS$. As in Sec.\,\ref{S:NormalD}, we deduce

\begin{coro}\label{C:ComplexS}
Assume that $\MM$ is a left-cancellative monoid with no nontrivial invertible element and $\SSS$ is a Garside family in~$\MM$. Then $\SSS$-normal decompositions can be computed in linear space and quadratic time. The Word Problem for~$\MM$ with respect to~$\SSS$ lies in~$\textsc{dtime}(\nn^2)$.
\end{coro}

On the other hand, as the diagram of Fig.\,\ref{F:LeftMultD} remains valid, we obtain

\begin{coro}\label{C:LeftFTPS}
Assume that $\MM$ is a left-cancellative monoid with no nontrivial invertible element and $\SSS$ is a Garside family in~$\MM$. Then the $\SSS$-normal words satisfy the $2$-Fellow traveller Property on the left.
\end{coro}

In contrast to the particular case of Garside monoids, there is no symmetric counterpart involving right multiplication in the framework of an arbitrary Garside family. As can be expected, the existence of such a counterpart is equivalent to the validity of the right domino rule for $\SSS$-normal words of length~two. 

\noindent\begin{minipage}{\textwidth}
\rightskip44mm
Now, the \VR(3.2,0) latter may fail, as the counterexample on the right shows: here $\SSS$ is the sixteen-element Garside family described in  Fig.\,\ref{F:Artin} in the Artin--Tits monoid of type~$\Att$.\hfill
\hfill\begin{picture}(0,0)(-6,-1.5)
\psarc[style=thin](15,0){3}{180}{360}
\psarc[style=thin](0,0){3.5}{0}{90}
\psarc[style=thin](15,0){3.5}{0}{90}
\psarc[style=thin](15,10){3}{0}{180}
\pcline{->}(1,0)(14,0)\tbput{$\sig2$}
\pcline{->}(16,0)(29,0)\tbput{$\sig3$}
\pcline{->}(1,10)(14,10)\taput{$\sig1$}
\pcline{->}(16,10)(29,10)\taput{$\sig1\sig2$}
\pcline{->}(0,9)(0,1)\trput{$\sig1\sig2\sig1$}
\pcline{->}(15,9)(15,1)\trput{$\sig1\sig2\sig1$}
\pcline{->}(30,9)(30,1)\trput{$\sig1\sig3$}
\psline[style=thin](14,-4)(16,-2)
\psline[style=thin](14,-2)(16,-4)
\end{picture}
\end{minipage}

\rightskip0mm
For \VR(6,0) more results on the question, we refer to Chap.\,V of~\cite{Garside}, where the notion of a \emph{bounded} Garside family is introduced, and where it is proved that the right domino rule and the counterpart of Prop.\,\ref{P:RecipeRD} are valid, whenever $\SSS$ is a bounded Garside family.

\pagebreak

\subsection{Existence of $\SSS$-normal decompositions}\label{SS:Rec}

The notion of a Garside family is useful only if we can provide practical characterisations, which amounts to giving sufficient conditions for $\SSS$-normal decompositions to exist. A number of such characterisations are known \cite[Chap.\,IV]{Garside}, and we shall only mention a few of them.

Two types of characterisations exist, according to whether the ambient mon\-oid satisfies or not additional conditions. Let us begin with the case of a monoid that is just assumed to be left-cancellative and, in this paper, to admit no nontrivial invertible element. In order to state the results, we need two definitions.

\begin{defi}
\rm If $\MM$ is a left-cancellative monoid, $\SSS$ is included in~$\MM$, and $\gg$ is an element of~$\MM$, then an element~$\sss$ of~$\SSS$ is said to be an \emph{$\SSS$-head} of~$\gg$ if we have $\sss \dive \gg$ and $\forall\ttt{\in}\SSS\,(\ttt \dive \gg \Rightarrow \ttt \dive \sss)$.
\end{defi}

In other words, an $\SSS$-head of~$\gg$ is a greatest left-divisor of~$\gg$ lying in~$\SSS$. An $\SSS$-head is unique whenever the ambient monoid~$\MM$ has no nontrivial invertible element: if $\sss$ and~$\sss'$ are $\SSS$-heads of~$\gg$, the definition implies $\sss \dive \sss'$ and $\sss' \dive \sss$, whence $\sss' = \sss$. If $(\MM, \Delta)$ is a Garside monoid, the $\Div(\Delta)$-head of an element~$\gg$ exists and is simply the left-gcd of~$\gg$ and~$\Delta$, as considered in Lemma~\ref{L:Head}.

\noindent\begin{minipage}{\textwidth}
\begin{defi}\label{D:RMClosed}
\rightskip41mm
If $\MM$ is \VR(8,0) a left-cancellative monoid and $\SSS$ is included in~$\MM$, we say that $\SSS$ is \emph{closed under right-comultiple} if the relation\\
\begin{picture}(0,10)(0,-1)
\put(7,4.5){$\forall\sss, \ttt {\in} \SSS\ \forall\gg{\in} \MM \, ((\sss \dive \gg \text{\ and\ } \ttt \dive \gg)$}
\put(20,0){$\Rightarrow \exists \rr {\in} \SSS\, (\sss \dive\rr  \text{\ and\ } \ttt \dive \rr \text{\ and\ } \rr \dive \gg))$}
\end{picture}\\
is satisfied in~$\MM$, as illustrated on the right.\hfill%
\begin{picture}(0,0)(-8,-2)
\pcline{->}(1,20)(14,20)
\taput{$\ttt$}
\pcline{->}(0,19)(0,11)
\tlput{$\sss$}
\pcline(0,9)(0,7)
\psarc(7,7){7}{180}{270}
\pcline{->}(7,0)(29,0)
\pcline(16,20)(23,20)
\psarc(23,13){7}{0}{90}
\pcline{->}(30,13)(30,1)
\pcline[style=exist]{->}(1,10)(14,10)
\pcline[style=exist]{->}(15,19)(15,11)
\pcline[style=exist]{->}(16,9)(29,0.5)
\pcline[style=exist]{->}(1,19)(14,11)
\put(7,16){$\rr{\in}\SSS$}
\end{picture}
\end{defi}
\end{minipage}

Thus, \VR(8,0) a family~$\SSS$ is closed under right-comultiple if every common right-multiple of two elements~$\sss, \ttt$ of~$\SSS$ is a right-multiple of some common right-multiple of~$\sss$ and~$\ttt$ that lies in~$\SSS$. Finally, we naturally say that a family~$\SSS$ is \emph{closed under right-divisor} if every right-divisor of an element of~$\SSS$ belongs to~$\SSS$.

\begin{prop}\label{P:RecGarII} {\rm\cite[Prop.\,3.9]{Dif} or \cite[Prop.\,IV.1.24]{Garside}}
Assume that $\MM$ is a left-cancellative monoid with no nontrivial invertible element, and $\SSS$ is a generating subfamily of~$\MM$ that contains~$1$. Then $\SSS$ is a Garside family in~$\MM$---that is, every element of~$\MM$ admits an $\SSS$-normal decomposition---if, and only if, it satisfies one of the following equivalent conditions:

\ITEM1 Every nontrivial element of $\MM$ admits an $\SSS$-head, and $\SSS$ is closed under right-divisor.

\ITEM2 Every element of $\SSS^2$ admits a $\divs$-maximal left-divisor in~$\SSS$, and $\SSS$ is closed under right-comultiple and 
right-divisor.
\end{prop}

The conditions of Prop.\,\ref{P:RecGarII} are not demanding: very little is required for the existence of $\SSS$-normal decompositions. The difference between~\ITEM1 and~\ITEM2 is that, in~\ITEM2, the existence of an $\SSS$-head is relaxed twice: one considers elements of~$\SSS^2$ (that is, elements that can be expressed as the product of two elements of~$\SSS$) rather than arbitrary elements, and $\divs$-maximal left-divisors, which is weaker than $\dive$-greatest left-divisors, since it amounts to replacing $\sss \dive \ttt$ with $\ttt \not\divs \sss$.

\begin{exam}
\rm Prop.\,\ref{P:RecGarII}\ITEM1 makes it straightforward that, if $(\MM, \Delta)$ is a Garside monoid, $\Div(\Delta)$ is a Garside family: as noted above, the left-gcd of~$\gg$ and~$\Delta$ is a $\Div(\Delta)$-head of~$\gg$ and, by definition, $\Div(\Delta)$ is closed under right-divisor.

The argument is similar for the family $\Div(\Delta)$ in the Klein bottle monoid~$K^+$ of Example~\ref{X:GarFam}, but one easily finds examples of a completely different flavour. For instance, the reader can play with the family $\{\ttb^\ii \mid 0 \le \ii \le \nn+1\} \cup \{\tta\}$ in the monoid $\MON{\tta, \ttb}{\tta\ttb^\nn = \ttb^{\nn+1}}$ with $\nn \ge 1$.  
\end{exam}

When the ambient monoid satisfies additional assumptions, the conditions of Prop.\,\ref{P:RecGarII} can still be weakened.

\begin{defi}
\rm A left-cancellative monoid is called \emph{right-noetherian} if there is no infinite descending sequence with respect to strict right-divisibility.
\end{defi}

Equivalently, a left-cancellative monoid is right-noetherian if, and only if, there is no infinite bounded ascending sequence with respect to strict left-divisibility, meaning that $\gg_1 \divs \gg_2 \divs \pdots \dive \gg$ is impossible. When a monoid is right-noetherian, the existence of $\divs$-maximal elements is for free, and we deduce

\enlargethispage{7mm}

\begin{coro}\label{C:Noeth}
Assume that $\MM$ is a right-noetherian left-cancellative monoid with no nontrivial invertible element, and $\SSS$ is a generating subfamily of~$\MM$ that contains~$1$. Then $\SSS$ is a Garside family if, and only if, $\SSS$ is closed under right-comultiple and right-divisor.
\end{coro}

The criterion can be further improved as, for the ambient monoid to be right-noetherian, it is sufficient that the restriction of right-divisibility to the considered family~$\SSS$ is, a trivial condition when $\SSS$ is finite, see~\cite[Prop.\,IV.2.18]{Garside}.

Finally, things become even more simple when the ambient monoid admits \emph{conditional right-lcms}, meaning that any two elements that admit a common right-multiple admit a right-lcm. Then closure under right-comultiple boils down to closure under right-lcm (that is, the right-lcm of two elements of~$\SSS$ belongs to~$\SSS$ when it exists), and we obtain

\begin{coro}\label{C:Lcm}
Assume that $\MM$ is a left-cancellative monoid that is right-noeth\-er\-ian, admits conditional right-lcms, and contains no nontrivial invertible element, and $\SSS$ is a generating subfamily of~$\MM$ that contains~$1$. Then $\SSS$ is a Garside family  if, and only if, $\SSS$ is closed under right-lcm and right-divisor.
\end{coro}

Thus, in the context of Cor.\,\ref{C:Lcm}, being a Garside family is a closure property. It follows that, for every generating set~$\At$, there exists a smallest Garside family~$\SSS$ that includes~$\At$, namely the closure of~$\At$ under right-lcm and right-divisor. When the ambient monoid is noetherian (meaning left- and right-noetherian), it admits a smallest generating family, namely the family of atoms (indecomposable elements), and therefore it admits a smallest Garside family, the closure of atoms under right-lcm and right-divisor. A typical example is the family~$\Div(\Delta)$ in a Garside monoid (with $\Delta$ chosen minimal), but another example is the Garside family of Fig.\,\ref{F:Artin} in the Artin--Tits monoid of type~$\Att$.

\pagebreak

Let us mention a last result. We observed that the definition of an $\SSS$-normal sequence in~\eqref{E:NormalS} is a priori more demanding than that of~\eqref{E:NormalD}. It turns out that, when $\SSS$ satisfies convenient conditions, the conditions become equivalent:

\begin{prop} {\rm\cite[Prop.\,IV.1.20]{Garside}} \label{P:Connection2}
Assume that $\MM$ is a left-cancellative mon\-oid with no nontrivial invertible element, and $\SSS$ is a generating family of~$\MM$ that is closed under right-comultiple and right-divisor. Then an $\SSS$-word $\sss_1 \sep \pdots \sep \sss_\pp$ is $\SSS$-normal if, and only if, it satisfies the condition
\begin{equation}
\forall \sss {\in}Ê\SSS \  (\, \sss_\ii \divs \sss \ \Rightarrow\ \sss \notdive \sss_\ii\sss_{\ii+1}\pdots \sss_\pp\, ).
\end{equation}
\end{prop}

We already observed that the condition is necessary. That it is sufficient follows from arguments extending those of Ex.\,\ref{X:Connection}. Note that, by Prop.\,\ref{P:RecGarII}, every Garside family satisfies the assumptions of Prop.\,\ref{P:Connection2} and, therefore, the connection is valid in this case.

\begin{rema}
\rm Once again, we can think of extending the results from the monoid to its enveloping group. Here, some care is needed as, in general, a left-cancellative monoid (even a cancellative one) need not embed in a group of left fractions: by the classical Ore theorem, this happens if, and only if, the monoid~$\MM$ is cancellative and any two elements of~$\MM$ admit a common left-multiple. But, even in this case, the existence of unique $\SSS$-normal decompositions in~$\MM$ does not directly lead to unique distinguished decompositions for the elements of its group of fractions, because fractional decompositions need not be unique. However, when the monoid~$\MM$ admits left-lcms, a notion of irreducible fraction arises and one obtains unique decompositions (called ``symmetric $\SSS$-normal'')  for the elements of the group by using $\SSS$-normal decompositions for the numerator and the denominator of an irreducible fractional decomposition, see~\cite[Sec.\,III.2]{Garside}.
\end{rema}

\section{Quadratic normalisation}\label{S:NormalQ}

Proceeding one step further, we now consider more general normalisation processes that include those of the previous sections, but also new examples of a different flavour. However, we shall see that the mechanism of Garside normalisation, as captured in Prop.\,\ref{P:RecipeS}, can be retrieved in the more general framework of what we shall call ``quadratic normalisations of class~$(4,3)$''. One of the benefits of such an extended approach is that some monoids that are not even left-cancellative, like plactic monoids, become in turn eligible.

\subsection{Normalisations and geodesic normal forms}\label{SS:NormSys}

We now restart from a general standpoint and consider (not necessarily cancellative) monoids equipped with a generating family. We are interested in normal forms in such monoids, according to the following abstract scheme:

\begin{defi}\label{D:NF}
\rm Assume that $\MM$ is a monoid and $\SSS$ is a generating subfamily of~$\MM$. A \emph{normal form on $(\MM,\SSS)$} is a (set-theoretic) section of the canonical projection~$\ev$ of~$\SSS^*$ onto~$\MM$. A normal form~$\NF$ on $(\MM,\SSS)$ is called \emph{geodesic} if, for every~$\gg$ in~$\MM$, we have $\LG{\NF(\gg)} \le \LG\ww$ for every $\SSS$-word~$\ww$ representing~$\gg$. 
\end{defi}

Typically, we saw in Sec.\,\ref{S:NormalS} that, if $\MM$ is left-cancellative with no nontrivial invertible element, every Garside family~$\SSS$ of~$\MM$ provides a normal form on~$(\MM, \SSS)$, associating with every element~$\gg$ of~$\MM$ the unique strict $\SSS$-normal decompositon of~$\gg$. This normal form is geodesic, since, by Prop.\,\ref{P:RecipeS},  the $\SSS$-normal form of an element specified by an $\SSS$-word of length~$\pp$ has length at most~$\pp$. 

As already done in Sec.\,\ref{S:NormalD} and~\ref{S:NormalS}, we shall rather work with words, and concentrate on the normalisation maps that associate to an arbitrary word the unique equivalent normal word. This leads us to the following notion.

\begin{defi}\label{D:NormSys}
\rm A \emph{normalisation} is a pair $(\SSS,\NM)$, where~$\SSS$ is a set and~$\NM$ is a map from~$\SSS^*$ to itself satisfying, for all $\SSS$-words~$\uu, \vv, \ww$,
\begin{gather}
\label{E:NormSys1}
\LG{\NM(\ww)} = \LG\ww, \\
\label{E:NormSys2}
\LG\ww = 1 \text{\ implies\ } \NM(\ww) = \ww, \\
\label{E:NormSys3}
\NM(\uu \sep \NM(\ww) \sep \vv)=\NM(\uu \sep \ww \sep \vv).
\end{gather}
An $\SSS$-word~$\ww$ satisfying $\NM(\ww)=\ww$ is called \emph{$\NM$-normal}. If~$\MM$ is a monoid, we say that $(\SSS, \NM)$ is a normalisation \emph{for~$\MM$} if~$\MM$ admits the presentation
\begin{equation}\label{E:NormSys4}
\MON{\SSS}{\{\ww = \NM(\ww) \mid \ww \in \SSS^*\}}.
\end{equation}
\end{defi}

Note \VR(5,0) that~\eqref{E:NormSys3} implies that~$\NM$ is idempotent. We shall see below that the maps~$\NMD$ and~$\NMS$ considered in Sec.\,\ref{S:NormalD} and~\ref{S:NormalS} are typical examples of normalisations. Many others appear in~\cite{Dip}. The connection between normalisations and normal forms is easily described, especially in the case when all equivalent $\SSS$-words have the same length.

\begin{prop}\label{P:NF} {\rm\cite[Prop.\,2.1.12]{Dip}}
If $(\SSS,\NM)$ is a normalisation for a mon\-oid~$\MM$, then putting $\NF(\gg) = \NM(\ww)$, where $\ww$ is any $\SSS$-decomposition of~$\gg$, provides a normal form on~$(\MM,\SSS)$.

Conversely, if $\MM$ is a monoid, $\SSS$ is a generating subfamily of~$\MM$, and $\NF$ is normal form on~$(\MM,\SSS)$, and, moreover, any two $\SSS$-decompositions of an element of~$\MM$ have the same length, then putting $\NM(\ww) = \NF(\ev(\ww))$ provides a normalisation for~$\MM$.
\end{prop}

Moreover, it is easily seen that the two correspondences of Prop.\,\ref{P:NF} are inverses of one another.

When the elements of~$\MM$ may admit $\SSS$-decompositions of different lengths (as in the case of a Garside family), more care is needed, but we can still merge unique normal forms and length-preserving normalisation maps at the expense of introducing a dummy letter that represents~$1$ and is eventually collapsed.

\begin{defi}\label{D:Neutral}
\rm If $(\SSS, \NM)$ is a normalisation, an element~$\ee$ of~$\SSS$ is called \emph{$\NM$-neutral} if, for every~$\SSS$-word~$\ww$, one has \begin{equation}\label{E:Neutral1}
\NM(\ww \sep \ee) = \NM(\ee \sep \ww) = \NM(\ww) \sep \ee.
\end{equation}
Then we write~$\coll_\ee$ for the action of erasing~$\ee$ in an $\SSS$-word. If $\MM$ is a monoid, we say that $(\SSS, \NM)$ is a normalisation \emph{for~$\MM$ mod~$\ee$} if $\MM$ admits the presentation 
\begin{equation}\label{E:Neutral2}
\MON{\SSS}{\{\ww = \NM(\ww) \mid \ww \in \SSS^*\} \cup \{\ee = 1\}}.
\end{equation}
\end{defi}

We invite the reader to check that, if $\SSS$ is a Garside family in a left-cancellative monoid~$\MM$ that admits no nontrivial invertible element, then $(\SSS, \NMS)$ is a normalisation for~$\MM$ mod~$1$ and the $\SSS$-normal words of Sec.\,\ref{S:NormalS} are the associated $\NMS$-normal words. 

\pagebreak

Then Prop.\,\ref{P:NF} extends in

\begin{prop}\label{P:GenNF} {\rm\cite[Prop.\,2.2.7]{Dip}}
If $\MM$ is a monoid and $(\SSS,\NM)$ is a normalisation for~$\MM$ mod~$\ee$,  putting $\NF(\gg) = \coll_\ee(\NM(\ww))$, where $\ww$ is any $\SSS$-decomposition of~$\gg$, provides a geodesic normal form on~$(\MM,\SSS \setminus \{\ee\})$.

Conversely, if $\MM$ is a monoid, $\SSS$ generates~$\MM$, and~$\NF$ is a geodesic normal form on~$(\MM,\SSS)$, putting $\NM(\ww) = \NF(\ev(\ww))  \sep \ee^\mm$, with $\ee$ a new letter not in~$\SSS$ evaluated to~$1$ in~$\MM$ and $\mm = \LG\ww - \LG{\NF(\ev(\ww))}$, provides a normalisation for~$\MM$ mod~$\ee$ (with alphabet $\SSS \cup \{\ee\})$.
\end{prop}

Again, the two correspondences of Prop.\,\ref{P:GenNF} are inverses of one another. Thus, investigating geodesic normal forms and investigating normalisations are one and the same question.

This general framework being set, we now turn to a more specific situation. By Prop.\,\ref{P:RegularS} and \ref{P:RecipeS}, subfactors of length~two play a prominent r\^ole in Garside normalisation. This is the property we shall extend. We recall Notation~\ref{N:Positions}, in particular that, for $\NM: \SSS^* \to \SSS^*$, we use $\NMbar$ for the restriction of~$\NM$ to~$\Pow\SSS2$. 

\begin{defi}\label{D:Quad}
\rm A normalisation $(\SSS,\NM)$ is \emph{quadratic} if the two conditions hold:

\ITEM1 An $\SSS$-word~$\ww$ is $\NM$-normal if, and only if, every length-two factor of~$\ww$ is.

\ITEM2 For every $\SSS$-word~$\ww$, there exists a finite sequence of positions~$\uu$, depending on~$\ww$, such that $\NM(\ww)$ is equal to~$\NMbar_\uu(\ww)$.
\end{defi}

\begin{exam}\label{X:Normal}
\rm By Prop.\,\ref{P:RegularS} and \ref{P:RecipeS}, if $\SSS$ is a Garside family, then the associated normalisation $(\SSS, \NMS)$ is quadratic: Prop.\,\ref{P:RecipeS} says that $\uu:= \delta_\pp$ can be chosen for \emph{every} $\SSS$-word~$\ww$ of length~$\pp$.

For a different example, as in Ex.\,\ref{X:Abelian}, let $(\At_\nn, \NM^\Lex)$ be the lexicographic normalisation for the free abelian monoid~$\NNNN^\nn$. Then $(\At_\nn, \NM^\Lex)$ is quadratic. Indeed, an $\At_\nn$-word is $\NM^\Lex$-normal if, and only if, all its length-two subfactors are~$\tta_\ii \sep \tta_\jj$ with $\ii \le \jj$. On the other hand, every $\At_\nn$-word can be transformed into a $\NM^\Lex$-normal word by switching adjacent letters that are not in the due order.
\end{exam}

Simple counterexamples show that none of the two conditions in Def.\,\ref{D:Quad} implies the other. When a normalisation~$(\SSS, \NM)$ is quadratic, then, by definition, the restriction~$\NMbar$ of~$\NM$ to~$\Pow\SSS2$ is crucial and most properties can be read from~$\NMbar$. In particular, one shows that, if $\MM$ is a monoid and $(\SSS,\NM)$ is a quadratic normalisation for~$\MM$ (\resp for~$\MM$ mod~$\ee$), then $\MM$ admits the presentation 
\begin{gather}\label{E:QuadPres}
\MON{\SSS}{\{\sss\sep\ttt = \NMbar\HS{0.3}(\sss\sep\ttt) \mid \sss, \ttt \in \SSS\}},\\
(\resp \MON{\SSS \setminus \{\ee\}}{\{\sss\sep\ttt = \coll_{\ee} (\NMbar\HS{0.3}(\sss\sep\ttt)) \mid \sss, \ttt \in \SSS \setminus \{\ee\}\}}\ ). 
\end{gather}
So, the relations between words of length two bear all information.

Before turning to more elaborate results, let us immediately note the following direct consequence of Def.\,\ref{D:Quad}\ITEM1:

\begin{prop}\label{P:RegularQ}
If $(\SSS, \NM)$ is a quadratic normalisation and $\SSS$ is finite, then $\NM$-normal words form a regular language.
\end{prop}

\pagebreak

\subsection{The class of a quadratic normalisation}\label{SS:Class}

We now introduce a parameter, called the class, evaluating the complexity of the normalisation process associated with a quadratic normalisation.

If $(\SSS, \NM)$ is a quadratic normalisation and $\ww$ is an $\SSS$-word, $\NM(\ww)$ is obtained by successively applying the restriction~$\NMbar$ of~$\NM$ to~$\Pow\SSS2$ at various positions. So, in particular, for~$\LG\ww= 3$, we have $\NM(\ww) = \NMbar_\uu(\ww)$ for some finite sequence~$\uu$ of positions~$1$ and~$2$. As $\NMbar$ is idempotent, repeating~$1$ or~$2$ in~$\uu$ is useless, and it is enough to consider sequences~$\uu$ of the form~$121...$ or $212...$ (we omit separators). 

\begin{nota}
\rm For $\mm \ge 0$, we write $\ai{\mm}$ for the alternating sequence~$121...$ of length~$\mm$, and similarly for~$\aii\mm$. 
\end{nota}

So, if $(\SSS, \NM)$ is a quadratic normalisation, then, for every $\SSS$-word~$\ww$ of length three, there exists~$\mm$ such that $\NM(\ww)$ is $\NMbar_{\ai\mm}(\ww)$ or $\NMbar_{\aii\mm}(\ww)$.

\begin{defi}
\rm We say that a quadratic normalisation~$(\SSS, \NM)$ is \emph{of left class}~$\cc$ (\resp \emph{right-class}~$\cc$) if, for every~$\ww$ in~$\Pow\SSS3$, we have
$$\NM(\ww) = \NMbar_{\ai\cc}(\ww)\quad (\resp \NM(\ww) = \NMbar_{\aii\cc}(\ww)).$$
We say that $(\SSS, \NM)$ is \emph{of class~$(\cc, \cc')$} if it is of left class~$\cc$ and right class~$\cc'$.
\end{defi}

\begin{exam}\label{X:Abelian3} 
\rm If $(\MM, \Delta)$ is a Garside monoid, Prop.\,\ref{P:RecipeD} gives $\NMD(\ww) {=} \NMDbarr{212}\HS{1.5}(\ww)\VR(0,2)$ for every $\Div(\Delta)$-word~$\ww$ of length three. Hence, $(\Div(\Delta), \NMD)$ is of right class~$3$. Symmetrically, Prop.\,\ref{P:RecipeRD} gives $\NMD(\ww) = \NMDbarr{121}\HS{1.5}(\ww)\VR(3.5,2)$, so $(\Div(\Delta), \NMD)$ is also of left class~$3$. Hence, the normalisation $(\Div(\Delta), \NMD)$ is of class~$(3, 3)$. 

If $\SSS$ is a Garside family in a left-cancellative monoid with no nontrivial invertible element, then Prop.\,\ref{P:RecipeS} implies that $(\SSS, \NMS)$ is of right class~$3$, but, lacking in general a counterpart of Prop.\,\ref{P:RecipeRD}, we have no hint for the left class.

The reader can check that the lexicographic normalisation $(\At_\nn, \NM^\Lex)$ of Ex.\,\ref{X:Abelian} also has class~$(3, 3)$. On the other hand, there exist examples of normalisations with an arbitrarily high minimal class: see~\cite[Ex.\,3.3.9]{Dip}, where the minimal class is $(3 + \lfloor\log_2\nn\rfloor, 3 + \lfloor\log_2\nn\rfloor)$ for a size~$\nn$ alphabet. 

Let us mention one more normalisation, very different from the previous examples. If~$\XX$ is a linearly ordered finite set, the \emph{plactic} monoid over~$\XX$ is~\cite{CainGrayMalheiro}
\begin{equation*}\label{E:PlacticPresent}
P_X = \bigg\langle \XX \ \bigg\vert\ \  \begin{matrix}
\ttx\ttz\tty = \ttz\ttx\tty\  
&\text{for $\ttx\le\tty<\ttz$}\\
\tty\ttx\ttz = \tty\ttz\ttx\ 
&\text{for $\ttx<\tty\le\ttz$}
\end{matrix}
\ \smash{\bigg\rangle^{\!\!+}}.
\end{equation*}
Then $\PP_{\XX}$ is also generated by the family~$\Col$ of nonempty columns over~$\XX$, defined to be strictly decreasing products of elements of~$\XX$. Call a pair of columns $\sss_1\sep \sss_2$ \emph{normal} if $\LG{\sss_1}\ge\LG{\sss_2}$ holds and, for every $1\le\kk\le\LG{\sss_2}$, the $\kk$th element of~$\sss_1$ is at most the one of~$\sss_2$. Then normal sequences $\sss_1 \sep \pdots \sep \sss_\pp$ are in one-to-one correspondence with Young tableaux, and every element of~$P_X$ is represented by a unique tableau of minimal length (in terms of columns). Thus, mapping a $\Col$-word to the unique corresponding tableau defines a geodesic normal form on~$(\PP_{\XX},\Col)$. Writing $\Colp$ for $\Col$ enriched with one empty column~$\emptyset$ and using Prop.\,\ref{P:GenNF}, we obtain a normalisation $(\Colp, \NM)$ for~$P_X$ mod~$\emptyset$. Then, condition~\ITEM1 in Def.\,\ref{D:Quad} is satisfied by the definition of tableaux. Moreover, for every $\Colp$-word~$\ww$, the normal tableau~$\NM(\ww)$ can be computed by resorting to the Robinson--Schensted's insertion algorithm, progressively replacing each pair $\sss_1\sep\sss_2$ of subsequent columns by $\NMbar\HS{0.3}(\sss_1\sep\sss_2)$, which is a tableau with two columns (if the algorithm returns a tableau with one column, we insert an empty column to keep the length unchanged). So, the normalisation~$(\Colp,\NM)$ also satisfies condition~\ITEM2 in Def.\,\ref{D:Quad} and, therefore, it is quadratic. Then, the computations of~\cite[\S\S3.2--3.4 and \S\S4.2--4.4]{BokutChenChenLi} show that $(\Colp,\NM)$ is of class~$(3,3)$.
\end{exam}

There exists an easy connection between the left and the right class.

\begin{lemm}\label{L:Class}
If a quadratic normalisation is of left class~$\cc$, then it is of left class~$\cc'$ for every~$\cc'$ with $\cc' \ge \cc$, and of right class~$\cc''$ for every~$\cc''$ with $\cc'' \ge \cc+1$.
\end{lemm}

\begin{proof}
Assume that $(\SSS, \NM)$ is of left class~$\cc$. First, for~$\ww$ in $\Pow{\SSS}{3}$, the equality $\NM(\ww) = \NMbar_{\ai\cc}(\ww)$ implies $\NM(\ww) = \NMbar_{\ai{\cc+1}}(\ww)$, since $\NM(\ww)$ is $\NM$-invariant. So $(\SSS, \NM)$ is of left class~$\cc+1$ as well and, from there, it is of left class~$\cc'$ for~$\cc' \ge \cc$. 

On the other hand, we have $$\NM(\ww) = \NMbar_{\ai\cc}(\NMbar_2(\ww)) = \NMbar_{\aii{\cc+1}}(\ww)$$ by the assumption and by~\eqref{E:NormSys3}. Hence $(\SSS, \NM)$ is of right class~$\cc+1$ and, from there, of right class~$\cc''$ for every~$\cc''$ with $\cc'' \ge \cc+1$. 
\end{proof}

Hence, the minimal class of a quadratic normalisation~$(\SSS, \NM)$ is either~$(\cc, \cc')$ with $\vert\cc' - \cc\vert \le 1$, or~$(\infty, \infty)$, the latter being excluded for $\SSS$ finite. By Lemma~\ref{L:Class}, a Garside normalisation is of right class~$4$, and we can state:

\begin{prop}\label{P:GarClass}
If $\MM$ is a left-cancellative monoid with no nontrivial invertible element and $\SSS$ is a Garside family in~$\MM$, then the normalisation $(\SSS, \NMS)$ is of class~$(4, 3)$.
\end{prop}

\subsection{Quadratic normalisations of class $(4, 3)$}\label{SS:Class3}

We shall now show that many properties of Garside normalisations extend to all quadratic normalisations of class $(4,3)$. The extension comes from the following observation:

\begin{lemm}\label{L:DominoQ}
A quadratic normalisation~$(\SSS, \NM)$ is of class~$(4, 3)$ if, and only if, the left domino rule is valid for the family of all $\NM$-normal words of length~two.
\end{lemm}

\begin{proof}
Assume that~$(\SSS, \NM)$ is of right class~$3$, and let $\FFF$ be the family of all $\NM$-normal words of length~two. Let $\sss_1, \sss_2, \sss'_1, \sss'_2, \ttt_0, \ttt_1, \ttt_2$ be elements of~$\SSS$ satisfying the assumptions of Def.\,\ref{D:LeftDomino}. By the definition of the right class, we have $\NM(\ttt_0 \sep \sss_1 \sep \sss_2) = \NMbar_{212}(\ttt_0 \sep \sss_1 \sep \sss_2)$. As, by assumption, $\sss_1\sep\sss_2$ is $\NM$-normal, we obtain $$\NM(\ttt_0 \sep \sss_1 \sep \sss_2) = \NMbar_{12}(\ttt_0 \sep \sss_1 \sep \sss_2) = \NMbar_{2}(\sss'_1\sep\ttt_1\sep\sss_2) = \sss'_1\sep\sss'_2\sep\ttt_2.$$ So $\sss'_1\sep\sss'_2\sep\ttt_2$ is $\NM$-normal, hence so is $\sss'_1\sep\sss'_2$. Therefore, the left domino rule is valid for~$\FFF$. 

Conversely, assume that the left domino rule is valid for~$\FFF$. Let $\ttt_0 \sep \rr_1 \sep \rr_2$ belong to~$\Pow\SSS3$. Put 
$$\sss_1 \sep \sss_2 = \NMbar\HS{0.3}(\rr_1 \sep \rr_2), \quad \sss'_1 \sep \ttt_1 = \NMbar\HS{0.3}(\ttt_0 \sep \sss_1), \quad \text{and} \quad \sss'_2 \sep \ttt_2 = \NMbar\HS{0.3}(\ttt_1 \sep \sss_2).$$ Then $\sss'_2 \sep \ttt_2$ is $\NM$-normal by construction, and $\sss'_1 \sep \sss'_2$ is $\NM$-normal by the left domino rule, so $\sss'_1 \sep \sss'_2 \sep \ttt_2$ is $\NM$-normal. Hence we have $\NM(\ww) = \NMbar_{212}(\ww)$ for every~$\ww$ in~$\Pow\SSS3$. Therefore, $(\SSS, \NM)$ is of right class~$3$ and, therefore, of class~$(4, 3)$. 
\hfill\end{proof}

\pagebreak

Using the left domino rule exactly as in Sec.\,\ref{S:NormalD} and~\ref{S:NormalS}, we deduce

\begin{prop}\label{P:RecipeQ}
If $(\SSS, \NM)$ is a quadratic normalisation of class~$(4, 3)$, then, for every word~$\ww$ of length~$\pp$, we have 
\begin{equation}\label{E:RecipeQ}
\NM(\ww) = \NMbar_{\delta_\pp}(\ww).
\end{equation} 
\end{prop}

So the universal recipe given by the sequence of positions~$\delta_\pp$ is valid for every normalisation of class~$(4,3)$. Of course, a similar recipe associated with the sequence of positions~$\widetilde\delta_\pp$ as in Prop.\,\ref{P:RecipeRD} is valid for every normalisation of class~$(3,4)$, with the right domino rule replacing the left one. In the case of a normalisation of class~$(3, 3)$, both recipes are valid, as in the case of the Garside normalisation associated with a Garside monoid or, more generally, with a bounded Garside family.

Arguing exactly as in the previous sections, we deduce

\begin{coro}\label{C:ComplexQ}
If $(\SSS, \NM)$ is a quadratic normalisation of class~$(4, 3)$ for a mon\-oid~Ê$\MM$, then $\NM$-normal decompositions can be computed in linear space and quadratic time. The Word Problem for~$\MM$ with respect to~$\SSS$ lies in~$\textsc{dtime}(\nn^2)$, and, if $\ee$ is a $\NM$-neutral element of~$\SSS$ and $\MM_\ee$ is the quotient of~$\MM$ obtained by collapsing~$\ee$, so does the Word Problem for~$\MM_\ee$ with respect to~$\SSS \setminus \{\ee\}$.
\end{coro}

\begin{coro}\label{C:LeftFTPQ}
If $(\SSS, \NM)$ is a quadratic normalisation of class~$(4, 3)$ {\rm(}\resp of class~$(3,4)${\rm)}, then $\NM$-normal words satisfy the $2$-Fellow traveller Property on the left {\rm(}\resp on the right{\rm)}.
\end{coro}

Let us turn to another question, and mention (without proof) one further result. We start from the (easy) observation that the class of a normalisation~$(\SSS, \NM)$ can be characterised by algebraic relations satisfied by the map~$\NMbar$ and its translated copy.

\begin{lemm}\label{P:Class} {\rm\cite[Prop.\,3.3.5]{Dip}}
A quadratic normalisation~$(\SSS, \NM)$ is of left class~$\cc$ if, and only if, the map $\NMbar$ satisfies $\NMbar_{\ai\cc} = \NMbar_{\ai{\cc+1}} = \NMbar_{\aii{\cc+1}}$; it is of class~$(\cc, \cc)$ if, and only if, the map $\NMbar$ satisfies $\NMbar_{\ai\cc} = \NMbar_{\aii\cc}$.
\end{lemm}

So, in particular, if a normalisation $(\SSS, \NM)$ is of class~$(4, 3)$,  the map~$\NMbar$, which, by definition, is idempotent, satisfies $\NMbar_{212} = \NMbar_{2121} = \NMbar_{1212}$. The next result provides an axiomatisation of class~$(4, 3)$ normalisations: it shows that, conversely, every idempotent map satisfying the above relation necessarily stems from such a normalisation.

\begin{prop}\label{P:Axiom} {\rm\cite[Prop.\,4.3.1]{Dip}}
If~$\SSS$ is a set and~$\FF$ is a map from~$\Pow\SSS2$ to itself satisfying
\begin{equation}\label{E:Axiom}
\FF_{212} = \FF_{2121} = \FF_{1212},
\end{equation}
there exists a quadratic normalisation~$(\SSS, \NM)$ of class~$(4, 3)$ satisfying~$\FF=\NMbar$.
\end{prop}

The problem is to extend~$\FF$ into a map~$\FF^*$ on~$\SSS^*$ such that~$(\SSS,\FF^*)$ is a quadratic normalisation of class~$(4, 3)$. The idea of the proof is to take the recipe given by~\eqref{E:RecipeQ} as a definition, and to show that the resulting map has the expected properties. The result is not trivial, and there is no counterpart for higher classes.

The specific reasons why the result works for class~$(4, 3)$ are the algebraic properties of the monoid that admits the presentation 
\begin{equation*}
\bigg\langle \sig1 \wdots \sig{\pp-1} \ \bigg\vert\ 
\sigg\ii2 = \sig\ii \text{\ for $\ii \ge 1$,}\  \begin{matrix}
\sig\ii \sig j = \sig j \sig\ii 
&\text{for $\jj - \ii \ge 2$}\\
\sig j \sig\ii \sig j = \sig\ii \sig j \sig\ii \sig\jj = \sig j \sig\ii \sig\jj \sig\ii 
&\text{for $\jj = \ii+1$}
\end{matrix}
\ \smash{\bigg\rangle^{\!\!+}}.
\end{equation*}
This monoid is a sort of asymmetric version of a symmetric group (or rather of the corresponding Hecke algebra at $\qq = 0$), which is considered and used by A.\,Hess and V.\,Ozornova in~\cite{HeOz}, and investigated by D.\,Krammer in~\cite{KraArt}.

\subsection{Characterising Garside normalisations}\label{SS:CharGar}

We observed in Prop.\,\ref{P:GarClass} that every Garside normalisation is of class~$(4, 3)$, a result that is optimal in general, since the right domino rule fails for the finite Garside family of Fig.\,\ref{F:Artin}, implying that the associated normalisation is not of class~$(3, 3)$. Conversely, it is easy to see that the lexicographic normalisation of Ex.\,\ref{X:Normal}, which is of class~$(3, 3)$, hence a fortiori~$(4, 3)$, does not stem from a Garside family. So the question arises of characterising Garside normalisations among all normalisations of class~$(4, 3)$. The answer is simple.

\begin{defi}
\rm Assume that $(\SSS, \NM)$ is a (quadratic) normalisation for a mon\-oid~$\MM$. We say that $(\SSS, \NM)$ is \emph{left-weighted} if, for all $\sss, \ttt, \sss', \ttt'$ in~$\SSS$, the equality $\sss'\vert\ttt' = \NM(\sss \sep \ttt)$ implies $\sss \dive \sss'$ in~$\MM$. 
\end{defi}

Thus, a normalisation~$(\SSS, \NM)$ is left-weighted if, for every~$\sss$ in~$\SSS$, the first entry of~$\NM(\sss \sep \ttt)$ is always a right-multiple of~$\sss$ in the associated monoid.

\begin{prop}\label{P:GarCharac} {\rm\cite[Prop.\,5.4.3]{Dip}}
Assume that~$(\SSS, \NM)$ is a quadratic normalisation mod~$1$ for a monoid~$\MM$ that is left-cancellative and contains no nontrivial invertible element. Then the following are equivalent:

\ITEM1 The family~$\SSS$ is a Garside family in~$\MM$ and $\NM = \NMS$ holds.

\ITEM2 The normalisation~$(\SSS, \NM)$ is of class~$(4, 3)$ and is left-weighted.
\end{prop}

The implication \ITEM1 $\Rightarrow$ \ITEM2 is almost straightforward. 
Indeed, if $\SSS$ is a Garside family and $\sss'_1\vert\sss'_2 = \NMS(\sss_1 \sep \sss_2)$ holds, we have $\sss_1 \dive \sss'_1 \sss'_2$ with~$\sss \in \SSS$, so the assumption that $\sss'_1 \sep \sss'_2$ is $\SSS$-normal implies $\sss_1 \dive \sss'_1$. Hence $\NMS$ is left-weighted.

The converse implication is much more delicate. The main point is to show that $\SSS$ is a Garside family in~$\MM$, which is proved by establishing that $\SSS$ is closed under right-divisor and every element of the ambient monoid~$\MM$ has an $\SSS$-head, and then using Prop.\,\ref{P:RecGarII}\ITEM1.

\subsection{Connection with rewriting systems}\label{SS:Termination}

There exists a simple connection between normalisations as introduced above and rewriting systems. We refer to~\cite{DeJo} or~\cite{Klo} for basic terminology.

\begin{lemm}\label{P:QuadNormRewr}
If $(\SSS,\NM)$ is a quadratic normalisation for a monoid~$\MM$, then putting $\RR = \{\sss\sep\ttt \to \NMbar\HS{0.3}(\sss\sep\ttt) \mid \sss,\ttt\in\SSS, \sss\sep\ttt\neq\NMbar\HS{0.3}(\sss\sep\ttt)\}$ provides a rewriting system~$(\SSS, \RR)$ that is quadratic, reduced, normalising, confluent, and presents~$\MM$.

Conversely, if $(\SSS,\RR)$ is a quadratic, reduced, normalising, and confluent rewriting system presenting a monoid~$\MM$, putting 
$\NM(\ww) = \ww'$, where~$\ww'$ is the $\RR$-normal form of~$\ww$, provides a quadratic normalisation~$(\SSS, \NM)$ for~$\MM$.

The above correspondences  are inverses of one another.
\end{lemm}

\begin{exam}\label{X:AbelianRewr}
\rm If $(\At_\nn,\NM^\Lex)$ is the lexicographic normalisation for the free commutative monoid~$\NNNN^\nn$ , the associated quadratic rewriting system~$(\At_\nn, \RR_\nn)$ consists of the $\nn(\nn-1)/2$ rules $\tta_\ii \sep \tta_\jj \to \tta_\jj \sep \tta_\ii$ for $1 \le \jj < \ii \le \nn$. 
\end{exam}

The correspondence of Lemma.\,\ref{P:QuadNormRewr} extends to  a normalisation mod a neutral letter~$\ee$ at the expense of defining a new system~$\RR_\ee$ by replacing $\sss\sep\ttt \to \NMbar\HS{0.3}(\sss\sep\ttt)$ with $\sss\sep\ttt \to \coll_{\ee}(\NMbar\HS{0.3}(\sss\sep\ttt))$, that is, of erasing the involved $\NM$-neutral letter.

By Lemma~\ref{P:QuadNormRewr}, a quadratic normalisation~$(\SSS,\NM)$ yields a reduced quadratic rewriting system~$(\SSS,\RR)$ that is normalising and confluent, meaning that, from every $\SSS$-word, there is a rewriting sequence leading to a $\NM$-normal word. This however does not rule out the possible existence of infinite rewriting sequences: the system~$(\SSS, \RR)$ need not a priori be terminating. Here again, class~$(4, 3)$ is the point where transition occurs.

\begin{prop}\label{P:43Terminating} {\rm\cite[Prop.\,5.7.1]{Dip}}
If $(\SSS, \NM)$ is a quadratic normalisation of \break class~$(4,3)$, then the associated rewriting system~$(\SSS, \RR)$ is terminating, and so is $(\SSS \setminus \{\ee\},\RR_\ee)$ if~$\ee$ is a $\NM$-neutral element of~$\SSS$. More precisely, every rewriting sequence from a length-$\pp$ word has length at most $2^\pp - \pp - 1$.
\end{prop}

The (delicate) proof consists in showing that every sequence of $\RR$-rewritings inevitably approaches a $\NM$-normal word: because of the left domino rule, in whatever order the rewritings are operated, the distance between the current word and its image under~$\NM$ cannot increase, and it must even decrease at some predictible intervals. 

Either by taking into account the influence of the right domino rule in the proof of the above result, or by an alternative direct argument based on the classical Matsumoto's lemma for the symmetric group~$\Sym_\pp$, one can show that, in the case of a normalisation of class~$(3,3)$, the upper bound $2^\pp - \pp - 1$ drops to~$\pp(\pp-1)/2$.

Owing to Prop.\,\ref{P:GarClass}, we obtain as a direct application of Prop.\,\ref{P:43Terminating}:

\begin{coro}\label{C:43Terminating}
Assume that $\MM$ is a left-cancellative monoid with no nontrivial invertible element and $\SSS$ is a Garside family in~$\MM$. Then the associated rewriting system is terminating. More precisely, every rewriting sequence from a length-$\pp$ word has length at most $2^\pp - \pp - 1$.
\end{coro}

By contrast, we have:

\begin{prop}\label{P:Termin44}
There exists a quadratic normalisation of class~$(4, 4)$ such that the associated rewriting system is not terminating.
\end{prop}

\begin{proof}[Proof (sketch)]
Let $\SSS := \{\tta, \ttb, \ttb', \ttb'', \ttc, \ttc', \ttc'', \ttd\}$ and let $\RR$ consist of the five rules 
$$\tta\ttb \to \tta\ttb',\ \ttb'\ttc' \to \ttb\ttc,\ \ttb\ttc' \to \ttb''\ttc'',\ \ttb'\ttc \to \ttb''\ttc'',\ \ttc\ttd \to \ttc'\ttd.$$
Then $(\SSS, \RR)$ is quadratic by definition, and the diagram of Fig.\,\ref{F:Cycle}, 
in which $\tta\ttb''\ttc''\ttd$ is $\RR$-normal, shows that $(\SSS, \RR)$ is not terminating, since it admits the length-$3$ cycle $\tta\ttb\ttc\ttd \to \tta\ttb'\ttc\ttd \to \tta\ttb'\ttc'\ttd \to \tta\ttb\ttc\ttd$. However, one can show (with some care) that $(\SSS, \RR)$ is normalising and confluent, and that the associated normalisation is of class~$(4, 4)$. 
\hfill\end{proof}

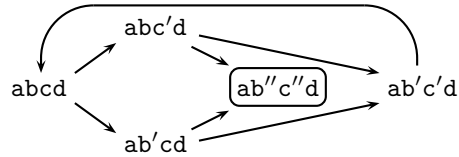
\begin{figure}[htb]
\pagebreak
$$\begin{picture}(55,20)(0,1)
\put(0,8){$\tta\ttb\ttc\ttd$}
\put(15,0.5){$\tta\ttb'\ttc\ttd$}
\put(15,16){$\tta\ttb\ttc'\ttd$}
\put(30,8){$\tta\ttb''\ttc''\ttd$}
\put(50,8){$\tta\ttb'\ttc'\ttd$}
\psframe[linewidth=0.8pt,framearc=.5](29,6.5)(42,11.8)
\psline{->}(8.5,7)(14,3)
\psline{->}(8.5,11)(14,15)
\psline{->}(24,3.5)(29,6)
\psline{->}(24,14.5)(29,12)
\psline{->}(25,2)(49,7)
\psline{->}(25,16)(49,11)
\psline(54,12)(54,13)
\psarc(47,13){7}{0}{90}
\psline(47,20)(11,20)
\psarc(11,13){7}{90}{180}
\psline{->}(4,13)(4,11)
\end{picture}$$
\caption{Non-termination of the rewriting system associated with the class $(4, 4)$ normalisation of Prop.\,\ref{P:Termin44}.}
\label{F:Cycle}
\end{figure}

It can be noted that terminating rewriting systems may also arise when the minimal class is~$(4,4)$: a beautiful example is provided by the Chinese monoid based on a set of size~$3$, see~\cite{CEKNH}.

\end{document}